\documentclass[12pt,reqno]{amsart}

\PassOptionsToPackage{colorlinks  = false,
  linkcolor   = black,
  citecolor   = black,
  urlcolor    = black}{hyperref}

\usepackage{hyperref}

\usepackage{amsthm,amssymb,bm,bbm,mathrsfs,mathtools,enumerate,graphicx,tikz,pgfplots,verbatim,textcomp,stmaryrd}
\usetikzlibrary{arrows.meta}
\usepackage[labelformat=simple,labelfont={}]{subcaption}
\usepackage[numbers,sort&compress]{natbib}

\usepackage{pifont}

\usepackage[T1]{fontenc}

\makeatletter
\def\mathcenterto#1#2{\mathclap{\protect\phantom{#1}\mathclap{#2}}\protect\phantom{#1}}
\let\old@widetilde\widetilde
\def\widetildeto#1#2{\mathcenterto{#2}{\old@widetilde{\mathcenterto{#1}{#2\,}}}}
\makeatother

\usepackage[margin=1.15in]{geometry}

\renewcommand\labelenumi{(\roman{enumi})}
\renewcommand\theenumi\labelenumi

\newcounter{mycount}

\renewcommand{\P}{\mathbb P}
\newcommand{\E}{\mathbb E}

\newcommand{\R}{\mathbb R}
\newcommand{\Z}{\mathbb Z}
\newcommand{\N}{\mathbb N}

\newcommand{\ve}{\varepsilon}
\newcommand{\eqd} {\overset{d}{=}}

\newcommand{\sA}{\mathscr A}

\newcommand{\sD}{\mathscr D}

\newcommand{\sG}{\mathscr G}

\newcommand{\sL}{\mathscr L}

\newcommand{\f}{\mathfrak}

\pgfplotsset{compat=newest}

\newcommand\dc\DeclareMathOperator
\newcommand\df\textbf

\newcommand\Bn{B_{t}}
\newcommand\m\mathbb
\renewcommand\c\mathcal
\renewcommand\f\frac

\newcommand\bbb[1]{\textnormal{\textlbrackdbl}{#1}\textnormal{\textrbrackdbl}}
\newcommand\bbr[1]{\textnormal{\textlbrackdbl}{#1}\rrparenthesis}
\newcommand\bbl[1]{\llparenthesis{#1}\textnormal{\textrbrackdbl}}
\newcommand\bbo[1]{\llparenthesis{#1}\rrparenthesis}

\dc\conv{conv}
\dc\col{col}
\dc\Ord{Ord}
\dc\Sym{Sym}
\dc\Col{Col}
\dc\inv{inv}
\dc\rec{rec}
\dc\joint{joint}
\dc\bub{bub}
\dc\numbub{\#\!\bub}
\dc\jt{jt}
\dc\MalCol{MalCol}
\dc\mallows{Mal}
\dc\Joint{Joint}
\dc\BMal{BMal}
\dc\Mal{Mal}
\dc\perm{perm}
\dc\Block{Block}
\dc\bubEnd{bubEnd}
\newcommand\sLtilde{\widetildeto{I}{\sL}}
\newcommand\sDtilde{\widetildeto{I}{\sD}}
\newcommand\graph\Gamma

\newcommand\ptcoln[1]{P_{t,q,#1}^{\col}}

\newcommand\ztqn[1]{Z(t,q,#1)}%{Z_{t,q,#1}}

\newcommand\dto[1]{\xrightarrow[#1]{d}}
\DeclareMathOperator*{\bigCircle}{\bigcirc}

\renewcommand\tilde\widetilde
\renewcommand\hat\widehat

\newtheorem{thm}{Theorem}
\newtheorem{prop}[thm]{Proposition}
\newtheorem{lemma}[thm]{Lemma}
\newtheorem{cor}[thm]{Corollary}

\newtheorem*{dfn*}{Definition}

\title{Mallows Permutations and Finite Dependence}

\date{28 June 2017}

\author[Alexander E. Holroyd]{Alexander E.\ Holroyd}
\address{Alexander E.\ Holroyd, Microsoft Research, Redmond, WA 98052, USA} \email{holroyd@microsoft.com}
\urladdr{\url{http://aeholroyd.org}}

\author{Tom Hutchcroft}
\address{Tom Hutchcroft, Department of Mathematics, University of British Columbia}
\email{tomhutchcroft@gmail.com}

\author{Avi Levy}
\address{Avi Levy, Department of Mathematics, University of Washington}
\email{avius@math.washington.edu}
\urladdr{\url{http://math.washington.edu/~avius}}


\begin{document}

\begin{abstract}
We use the Mallows permutation model to construct a new family of stationary finitely dependent proper colorings of the integers.  We prove that these colorings can be expressed as finitary factors of i.i.d.\ processes with finite mean coding radii.  They are the first colorings known to have these properties.  Moreover, we prove that the coding radii have exponential tails, and that the colorings can also be expressed as functions of countable-state Markov chains.  We deduce analogous existence statements concerning shifts of finite type and higher-dimensional colorings.
\end{abstract}

\maketitle

\section{Introduction}
  A stochastic process indexed by a metric space is said to be finitely dependent if subsets of variables separated by some fixed distance are independent. Finitely dependent processes appear in classical limit theorems, statistical physics, and probabilistic combinatorics \cite{alon2004probabilistic,erdos1975problems,heinrich1990asymptotic,hoeffding1948central,ibragimov1965independent,janson2015degenerate,liggett1997domination,o1980scaling}. For several decades the only known stationary finitely dependent processes were the block-factors: processes obtained from an iid sequence by applying a finite-range function. Indeed in 1965, Ibragimov and Linnik \cite{ibragimov1965independent,MR0322926} raised the question of whether there exist finitely dependent processes not expressible as block factors. While finitely dependent processes enjoyed significant attention in the intervening years \cite{MR1176437,MR972778,MR947821,de1989problem,de1993hilbert,gandolfi1989extremal,MR744235,janson1983renewal,ruschendorf1993regression}, it was only in 1993 that Ibragimov and Linnik's question was resolved in the affirmative by Burton, Goulet, and Meester \cite{burton1993}. Many subsequent works \cite{MR2721041,broman2005one,MR1269540,janson2015degenerate,liggett1997domination,matuvs1996two,matus1998combining} explored the properties of such processes, but the question remained: are there `natural' stationary finitely dependent processes that are not block-factors?

  Recently, Holroyd and Liggett \cite{HL} answered this question in the affirmative by proving the surprising fact that \emph{proper coloring} distinguishes between these classes of processes. More precisely, they constructed stationary finitely dependent colorings of $\Z$, and provided a simple argument showing that no block-factor is a coloring. A process $(X_i)_{i\in\Z}$ is a {\bf $\bm{q}$-coloring} if each $X_i$ takes values in $\{1,\ldots,q\}$, and almost surely $X_i\not=X_{i+1}$ for all $i$. A stationary $q$-coloring is {\bf $\bm{k}$-dependent} if the random sequences $(X_i)_{i<0}$ and $(X_i)_{i\geq k}$ are independent of one another. A process is finitely dependent if it is $k$-dependent for some $k$, and it is a coloring if it is a $q$-coloring for some~$q$. By an argument of Schramm \cite{HSW}, there is no stationary $1$-dependent $3$-coloring of $\Z$. Holroyd and Liggett \cite{HL} constructed a stationary $1$-dependent $4$-coloring and a stationary $2$-dependent $3$-coloring (implying trivially that stationary $k$-dependent $q$-colorings exist for all $k\geq 1$ and $q\geq 3$ other than $(k,q)=(1,3)$). These colorings were constructed in \cite{HL} by specifying cylinder probabilities (which are obtained in a rather mysterious way) and appealing to the Kolmogorov extension theorem, without a direct probabilistic construction on $\Z$.

  Here is a way to formalize this last concept. We say that $X=(X_i)_{i\in\m Z}$ is a {\bf finitary factor} of an iid process, or simply that $X$ is {\bf ffiid}, if it is equal in law to $F(Y)$ where $Y=(Y_i)_{i\in\m Z}$ is an iid sequence and $F$ is a {\bf translation-equivariant} function (i.e. one that commutes with translations of $\m Z$) satisfying the following property: for almost every sequence $y$ (with respect to the law of $Y$), there exists $r<\infty$ such that $F(y)_0=F(y')_0$ whenever $y'$ agrees with $y$ on $\{-r,\ldots,r\}$. Let $R(y)$ be the minimal such $r$. The random variable $R=R(Y)$ is the {\bf coding radius} of the finitary factor. In other words, $X_0$ is determined by examining only those variables $Y_i$ within a finite but random distance $R$ from the origin. Finitary factors generalize block-factors: the latter are finitary factors with bounded coding radius.

  In \cite{H} it was shown that the 1-dependent 4-coloring of \cite{HL} is  ffiid with infinite expected coding radius. However, the following question mentioned therein remained unanswered: does there exist a finitely dependent coloring that is ffiid with finite mean coding radius?

  We resolve this question as well as several others from \cite{H,HL,HL2} by constructing a new family of finitely dependent colorings whose coding radii have exponential tails.

  \begin{samepage}
  \begin{thm}\label{main}
    There exists a stationary, reversible, finitely dependent proper coloring of
    $\Z$ that is symmetric under permutations of the colors and that can be expressed in each of the following ways:
    \begin{enumerate}[(i)]
      \item as a finitary factor of an iid process, with exponential tail on the coding radius; and also
      \item as a function of a countable Markov chain with exponential tail on the return time to any given state.
    \end{enumerate}
    More precisely, there exists a stationary $k$-dependent $q$-coloring with all of the above properties for each of
    $$(k,q)=(1,5),\ (2,4),\ (3,3),$$
    as well as for all larger $q$ in each case.
  \end{thm}
  \end{samepage}

  By the statement that a process $X=(X_i)_{i\in\Z}$ is a `function of a countable Markov chain' we mean that there exists a stationary Markov chain $(Y_i)_{i\in\Z}$ on a countable state space $S$ and a function $h$ on $S$ such that $X$ has the same distribution as $(h(Y_i))_{i\in\Z}$. A process $X$ is {\bf reversible} if $(X_i)_{i\in\Z}$ has the same distribution as  $(X_{-i})_{i\in\Z}$, and a $q$-coloring $X$ is symmetric under permutations of the colors if $(X_i)_{i\in\Z}$ has the same distribution as $(\sigma(X_i))_{i\in\Z}$ for any permutation $\sigma$ of $\{1,\ldots,q\}$.

  Both the finitary factors and the countable Markov chains arising in Theorem \ref{main}  admit simple and explicit descriptions; see the Painting Algorithm later in the introduction.  See Figure~\ref{examples} for some simulations.  These descriptions involve an additional real parameter, $t$, which must be set at a specific, irrational value depending on $k$ and $q$ in order for the coloring to be finitely dependent.
\begin{figure}
\centering
\includegraphics[width=.95\textwidth]{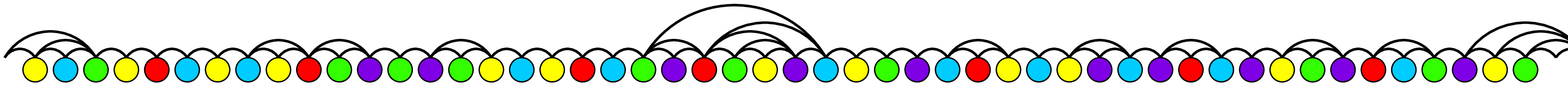}\\[5pt]
\includegraphics[width=.95\textwidth]{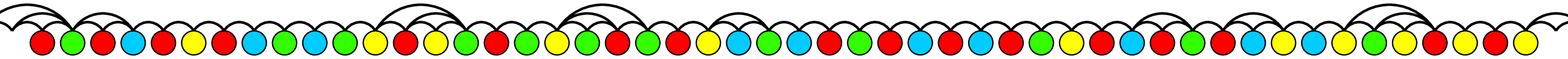}\\[5pt]
\includegraphics[width=.95\textwidth]{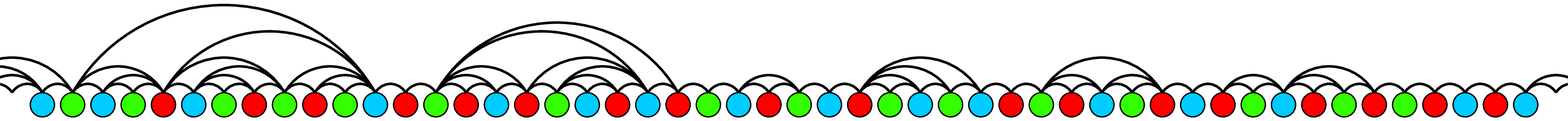}\\[5pt]
\caption{Samples of the $k$-dependent $q$-colorings from Theorem~\ref{main} for $(k,q)=(1,5),(2,4),(3,3)$ (top to bottom, respectively), together with the constraint graphs appearing in their construction.  The colors form a proper coloring of the constraint graph as well as of $\Z$.  The bubble endpoints (integers not covered by the interior of any one arc) form a Bernoulli process with carefully chosen irrational parameter depending on $k$ and $q$.  The constraint graph is not symmetric in law under reflection, even though the coloring is.}
\label{examples}
\end{figure}

  The existence of colorings satisfying properties (i) and (ii) of Theorem \ref{main} resolves open problems (v) and (ii) of \cite{HL}, respectively. The two properties are related for general stationary processes; we explain how in Theorem \ref{thm:bit}.

  We will show that our colorings with $k=1$ and $q\geq 5$ coincide with the
  colorings constructed in \cite{HL2}.
  The purpose of that paper was to show that there are symmetric finitely
  dependent $q$-colorings with $q\geq 5$.  % mysterious Chebyshev polynomials
  Thus Theorem~\ref{main} shows that these colorings are ffiid with finite mean coding radius, resolving \cite[Problem 4]{HL2}.  Moreover, the original colorings of \cite{HL} with $(k,q)=(1,4),(2,3)$ can be seen as boundary cases of our construction, as discussed later in the introduction.

  The $1$-dependent $q$-colorings from Theorem \ref{main} have the further
  property that if one conditions on the absence of color $q$, the resulting
  coloring is equal in law to the $2$-dependent $(q-1)$-coloring.
  (No other colorings from the theorem may be obtained from one another
  by conditioning in this manner.) Each of the $k$-dependent $q$-colorings
  constructed in Theorem \ref{main} is {\bf strictly $k$-dependent}, i.e.,
  $k$-dependent but not $(k-1)$-dependent.

  Sections 4 and 5 of \cite{HL} show that finitely dependent colorings can be
  written neither as block factors nor as functions of finite state Markov
  chains, respectively. The former result is a consequence of an earlier
  result in \cite{naor1991lower}, where it appears in a different form,
  motivated by applications in distributed computing. Further consequences
  and extensions appear in \cite{MR3262072} and \cite{HSW}. For example,
  block factors must contain arbitrarily long constant sequences with positive
  probability.

  In addition to showing that proper coloring distinguishes general stationary finitely dependent processes from block factors, Holroyd and Liggett \cite{HL} generalized this conclusion from proper coloring to all local constraints (i.e.\ shifts of finite type) satisfying a certain non-degeneracy condition. As a consequence of our main theorem, we will deduce an analogue of this result involving properties (i) and (ii) of the main theorem. See the discussion later in the introduction.

  We reiterate a natural conjecture suggested in \cite{H}: there exists
  a $k$-dependent $q$-coloring that is a finitary factor of an iid process
  with \textit{finite mean} coding radius if and only if $k\geq 1$, $q\geq 3$, and
  $(k,q)\not\in \{(1,3),(2,3),(1,4)\}$. Theorem \ref{main} establishes half
  of this conjecture: it remains to be seen whether every ffiid coloring has infinite mean coding radius in the remaining cases $(k,q)\in\{(2,3),(1,4)\}$.

  Coloring has applications in computer science.
  Colors may represent time schedules or communication frequencies
  for machines in a network, where adjacent machines are not permitted to conflict
  with each other. Finite dependence implies a security benefit --- an adversary
  who gains knowledge of some colors learns nothing about the others, except
  within a fixed finite distance. A ffiid coloring with finite mean coding
  radius is  desirable for the purpose of efficient computation.
  Such a coloring can be computed by the machines in distributed fashion,
  based on randomness generated locally, combined with communication with other machines within a random distance of finite mean. All machines follow the
  same protocol, and no central authority is needed. See e.g.
  \cite{linial1987distributive,naor1991lower} for more information.

\subsubsection*{Outline of proof}
We next discuss the main ideas behind the proof of Theorem~\ref{main}, which involves an intricate interplay of various ideas from combinatorics and physics.

At the heart of our construction (as well as those of \cite{HL}) is the following simple but mysterious picture.  Imagine that integers arrive in a \emph{random} order.  When an integer arrives, it is assigned a uniformly random color from those not present among its current neighbors, by which we mean the nearest integers to its left and right that arrived previously.  As a useful alternative description, the random order gives rise to a graph, which we call the \emph{constraint graph}, in which two integers are adjacent if and only if they were neighbors at some time.  (The constraint graph was also considered in \cite{H}, and may be interpreted as the planar dual of the binary search tree \cite{fill1996distribution} of a permutation.)  The final coloring is a uniformly random proper coloring of the constraint graph.

The proof of finite dependence begins with a version of this picture restricted to a finite interval, and involves remarkable cancellations that occur only when the set-up is precisely correct.  (Indeed, it is surprising that they can occur at all).  The required arrival order is not uniformly random.  Rather, it arises by re-weighting a simple underlying probability measure by the number of proper colorings of the constraint graph.  The fact that such a re-weighting can produce colorings with exceptional properties is reminiscent of the theory of two-dimensional quantum gravity, in which statistical mechanics models are studied on random planar maps that are weighted according their partition function for the model. There, as here, the model on the appropriately weighted random map has special properties that are not enjoyed by the same model on, say, a Euclidean lattice. See e.g. \cite{garban2012quantum} and references therein.

For the 1-dependent 4-coloring and 2-dependent 3-coloring of \cite{HL}, the underlying measure is uniform over permutations of an interval.  For our new construction, the underlying measure is the \emph{Mallows} measure, in which each permutation is weighted by a parameter $t$ raised to the power of the number of its inversions.  (An inversion is a pair of elements whose order is reversed.) The Mallows measure was originally introduced in statistical ranking theory \cite{mallows}, and has enjoyed a recent flurry of interest in contexts including mixing times \cite{benjamini2005mixing,diaconis2000analysis}, statistical physics \cite{starr2009thermodynamic,starr2015phase}, learning theory \cite{braverman2009sorting}, and longest increasing subsequences \cite{basu2016limit,MR3334280,mueller2013length}.  The computations and combinatorial identities  required to prove finite dependence in our case are $t$-analogues of those in \cite{HL}.  (The more usual terminology is `$q$-analogue', but in this article $q$ is reserved for the number of colors.  See \cite{stanley1997enumerative} for background on $q$-analogues.)  Since the Mallows measure is not reflection-invariant, the reversibility claimed in Theorem~\ref{main} requires a further highly non-trivial combinatorial argument.

The Mallows parameter $t$ must be chosen carefully.  Specifically, for the $q$-coloring to be $k$-dependent, the parameters $q$, $k$, and $t$ must satisfy the `tuning equation'
\begin{equation}\label{eqn:qkt}
  qt(1-t^k)=(1+t)(1-t^{k+1}).
\end{equation}
The tuning equation arises by setting a certain coefficient equal to zero in a recurrence for the cylinder probabilities of the colorings.
Finite dependence of the colorings stems from this cancellation.
This is reminiscent of a phenomenon in the theory of Schramm--Loewner evolution, in which SLE$(\kappa)$ curves possess additional distributional symmetries for special values of $\kappa$, stemming from cancellations in the coefficients of a stochastic differential equation \cite{lawler2008conformally}.

For the three cases $(k,q)=(1,5),(2,4),(3,3)$ highlighted in Theorem~\ref{main}, the required values of $t$ are respectively
  $$
    \frac{3-\sqrt{5}}{2},\quad \frac{3-\sqrt{5}}{2},\quad \text{and}\quad \frac{1+\sqrt{13}+\sqrt{2\bigl(\sqrt{13}-1\bigr)}}{4}.
  $$
The equality between the $t$ values for the pair of cases $(1,5)$ and $(2,4)$ generalizes to the pair $(1,q)$ and $(2,q-1)$ with $q\geq 4$.  This is behind the conditioning property mentioned earlier.

When restricted to finite intervals, the above construction yields a consistent family of random colorings, which extends to a coloring of $\m Z$ via Kolmogorov extension.
However, proving that this random coloring satisfies properties (i) and (ii) of the theorem requires a more direct construction.
To achieve this, we extend the random arrival picture to $\Z$.
This presents several challenges.
On a finite interval, the re-weighting introduces an extra factor every time an integer arrives at either end of the interval of its predecessors.

For the uniform model introduced in \cite{HL}, it turns out that these endpoint arrivals are sufficiently rare that their effect washes out in the limit, and the associated random order on $\Z$ is in fact uniform.  However, this means that the constraint graph has many long edges.  (A typical edge has infinite mean length, by the well-known record value waiting time property.)  This is the reason for the power law tail in the finitary factor construction of the 1-dependent 4-coloring in \cite{H}.  (Since it is also necessary to properly color the constraint graph, it turns out that this framework does not yield a finitary factor construction of the 2-dependent 3-coloring at all.  See \cite{H} and the earlier discussion.)

The situation for our model is very different.  For a fixed parameter $t$, the Mallows permutation of a sufficiently large finite interval can be naturally viewed as a perturbation of the identity, with a strong left-to-right bias in the corresponding order.  Consequently, (right) endpoint arrivals now have a positive density, and their re-weighting effect is \emph{not} washed out in the limit.  The resulting random order on $\Z$ follows a new (and quite natural) two-parameter variant of the Mallows measure, which we call the \emph{bubble-biased Mallows measure}.  (Infinite-interval versions of the standard Mallows measure were constructed in \cite{gnedin2012two}.)  As a result of the endpoint arrivals, the constraint graph is much better behaved than in the previous case.  It decomposes into a sequence of finite `bubbles', joined at their endpoints.  The length of a bubble has exponential tails, allowing us to prove properties (i) and (ii) in Theorem~\ref{main}.  Some further technical details are involved in making the transition from finite intervals to $\Z$ rigorous.  In particular, it is useful to consider convergence of the Lehmer code of a permutation (see e.g.\ \cite{beckenbach1964applied} for a definition).

The distinction between our new construction and that of \cite{HL} may be interpreted via the language of phase transition.  For the tuning equation \eqref{eqn:qkt} to have a solution in $t$, the parameters $k$ and $q$ must satisfy the inequality $qk\geq 2(k+1)$.  This is satisfied with equality along a critical curve $qk=2(k+1)$ in the $(k,q)$ plane -- see Figure~\ref{fig:phaseDiagram}.  On the curve we have $t=1$, and there are precisely two integer solutions, $(1,4)$ and $(2,3)$, giving the colorings of \cite{HL}.  (The Mallows measure reduces to the uniform measure when $t=1$.)  On one side of the curve, the construction does not work, while on the other side we obtain the colorings of this article.  This fits the signature of a phase transition: an abrupt qualitative change in behavior, with power laws at criticality, and exponential decay in the off-cricial regime.  But we believe that the same phase transition phenomenon applies to finitely dependent colorings of $\Z$ in complete generality, not just to the specific construction here (although currently no other constructions are known, besides trivial embellishments).  Indeed, it is proved in \cite{HL} that no stationary $1$-dependent $3$-coloring exists (so no solution exists on that side of the curve).  We believe that the stationary $1$-dependent $4$-coloring and $2$-dependent $3$-coloring (the `critical' cases) are unique.  (Some evidence for the former case is given in \cite{HL}.)  Moreover, we conjecture that no stationary $1$-dependent $4$-coloring or $2$-dependent $3$-coloring is ffiid with finite mean coding radius.

We reiterate that in \cite{H}, the $1$-dependent $4$-coloring was shown to be ffiid with infinite expected coding radius. We do not have an analogous explicit representation of the $2$-dependent $3$-coloring as a finitary factor of iid. A result of Smorodinsky \cite{smorodinsky1992finitary} states that stationary finitely dependent processes of equal entropy are finitarily isomorphic, implying that the 2-dependent 3-coloring of \cite{HL} is ffiid. Unfortunately, \cite{smorodinsky1992finitary} contains only a brief sketch of the proof, and the details of the argument do not seem to be available.

  \begin{figure}\label{figure:phaseDiagram}
    \tikz[domain=.3:7]{
      \begin{axis}[
        axis on top=false,
        axis x line=middle,
        axis y line=middle,
        stack plots=y,
        xmin = 0, xmax = 5.5,
        ymin = 2, ymax = 7.5,
        xlabel = {$k$},
        ylabel = {$q$},
        ytick = {2,3,4,5,6,7},
        y label style={left = 0},
        x label style={at={(current axis.right of origin)},anchor=north, below}
      ]
      \addplot+[mark=none, red, samples = 100] {2*(x+1)/x};
      \addplot+[mark=none, white, fill = blue!20, samples = 100] {2 * 1.3 / .3 - (2 * (x+1)/x)} \closedcycle;
      \end{axis}

      \begin{axis}[
        axis on top=false,
        axis x line=none, axis y line=none,
        xmin = 0, xmax = 5.5,
        ymin = 2, ymax = 7.5,
        x label style={at={(current axis.right of origin)},anchor=north, below=5mm},
        y label style={at={(current axis.above origin)},rotate=90,anchor=south east},
        xlabel = {$k$}, ylabel = {$q$},
        ytick = {2,3,4,5,6,7},
      ]
      \addplot[only marks, red, mark = square*] coordinates {
        (1,4) (2,3)
      };
      \addplot[only marks, blue] coordinates {
        (1,5) (1,6) (1,7)
        (2,4) (2,5) (2,6) (2,7)
        (3,3) (3,4) (3,5) (3,6) (3,7)
        (4,3) (4,4) (4,5) (4,6) (4,7)
        (5,3) (5,4) (5,5) (5,6) (5,7)
      };
      \draw [black, thick,rounded corners] (axis cs:.7,3.7) rectangle (axis cs:1.3,9);
      \addplot[only marks, black, mark = x, mark size = 3] coordinates {
        (1,3)
      };
      \addplot[mark = None]
        coordinates {(5,5)};
      \end{axis}
    }
    \caption{\label{fig:phaseDiagram}
      \small{ Phase diagram for $k$-dependent $q$-colorings. The phase boundary is the curve $qk=2(k+1)$. The two lattice points on this curve correspond to the $1$-dependent $4$-coloring and $2$-dependent $3$-coloring of \cite{HL}. In the region $qk<2(k+1)$ there do not exist $k$-dependent $q$-colorings (the $\times$ at $(k,q)=(1,3)$ indicates the non-trivial case ruled out by an argument of Schramm in \cite{HSW}). The $1$-dependent $q$-colorings from \cite{HL2} correspond to the outlined region. When $qk>2(k+1)$ there exists a $k$-dependent $q$-coloring that is ffiid with finite expected coding radius, by Theorem \ref{main}.
      }
    }
  \end{figure}
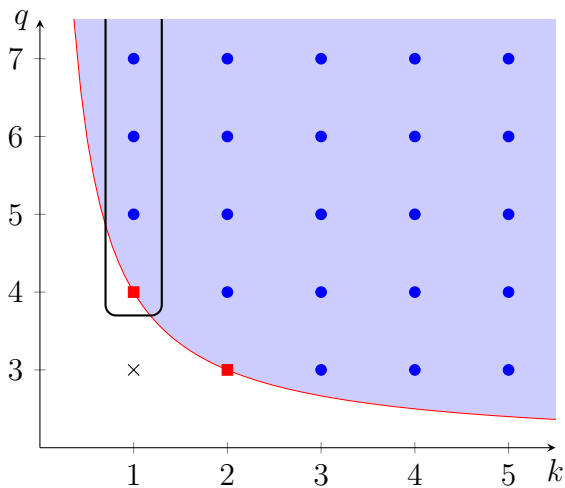

\subsubsection*{An algorithm for finitely dependent coloring.} There is a remarkably direct construction of the colorings from the main theorem, which we now present. While properties (i) and (ii) in the theorem (as well as stationarity, coloring constraints, and color symmetry) follow in a straightforward manner from the description, the finite dependence and reversibility properties are more subtle, and will be explained later.

%%%%%%%%%%%%%%%%%%%%%%%%%%%
%%%%%%%%%%%%%%%%%%%%%%%%%%%
%%                       %%
%%       Painting        %%
%%       algorithm       %%
%%                       %%
%%%%%%%%%%%%%%%%%%%%%%%%%%%
%%%%%%%%%%%%%%%%%%%%%%%%%%%

\medskip
\noindent\textbf{Painting Algorithm.}
Input: positive integers $q$ and $k$ satisfying $qk>2(k+1)$. Output: the $k$-dependent $q$-coloring $(X_i)_{i\in\m Z}$ constructed as follows.

\medskip\paragraph{\bf Stage 0.}  Let $t$ be the unique solution to \eqref{eqn:qkt} in $(0,1)$. Set $s=t(q-2)/(q-1-t)$. Start with $X_i$ unassigned for every $i\in \m Z$.

\medskip\paragraph{\bf Stage 1.}  Let $B=(B_i)_{i\in\Z}$ be an iid Bernoulli
process with each $B_i\in\{0,1\}$ taking value $1$ with probability $s$.
To each $i$ with $B_i=1$, assign a random color
$X_i\in\{1,\ldots,q\}$, in such a way that, conditional on $X$, the
subsequence $(X_i\colon B_i=1)$ of assigned values is the trajectory of a simple symmetric random walk on the complete graph with vertex set $\{1,\ldots,q\}$ at stationarity.

\medskip\paragraph{\bf Stage 2.}  Consider pairs of nearest integers $a<b$ that were assigned colors in Stage 1, that is, $B_a=B_b=1$ while $B_i=0$ for
$a<i<b$.  Independently for each such pair and conditional on Stage 1, we
fill in the missing colors via the following recursive procedure. Let $K$ be a random element of $\{a+1,\ldots,b-1\}$ with $\m P(K=\ell)=ct^\ell$, where $c$ is a constant of proportionality. Assign $X_K$ a uniformly random color in $\{1,\ldots,q\}\setminus \{X_a,X_b\}$. Conditional on the previous steps, recursively apply the same procedure to each of the intervals $\{a,\ldots,K\}$ and $\{K,\ldots,b\}$ until all integers have been assigned colors.

It is important to note that the conditional law, given Stage 1, of
the coloring $(X_i)_{i=a}^b$ for integers $a<b$ in Stage
2 is \emph{not} simply the conditional law of the final process
$(X_i)_{i\in\Z}$ restricted to $\{a,\ldots,b\}$ given $X_a$ and $X_b$.  For example, it is possible that $(X_1,X_2,X_3,X_4)=(1,2,1,2)$
(for instance if all $4$ values are assigned at Stage 1).  However, if Stage
1 assigns $X_1=1$ and $X_4=2$ but not $X_2$ or $X_3$, it is then impossible
for Stage 2 to fill the interval in this way, because both colors $1$ and $2$
are unavailable for the first insertion. Our construction is more subtle than
such na\"ive conditioning would suggest. In particular, the specific choices
of the parameters $s$ and $t$ of the Bernoulli and geometric processes are
crucial, as we shall see.

\subsubsection*{Bit-finitary factors of iid and countable Markov chains.}
In this section we investigate the connection between items (i) and (ii) of Theorem \ref{main}, that is, the connection between the expressibility of a process as a finitary factor of iid with certain desirable properties, and the expressibility of the same process as a function of a Markov chain with certain other desirable properties.

Rudolph \cite{rudolph1982mixing} proved that a bi-infinite trajectory of a mixing, countable, positive recurrent Markov chain satisfying a finite entropy condition can be expressed as a finitary factor of iid if and only if the time for the chain to hit any particular state has an exponential tail.
In this section, we provide a complement to Rudolph's result, giving a  sufficient condition under which a finitary factor process $X=F(U)$  of an iid process $U$ is also expressible as a function of a countable Markov chain. Intuitively, the condition is that $F$ has finite mean coding radius and that $F$ only has to query a finite (but random) number of the random bits of $U$ in order to compute each of its outputs.

For the purposes of this section, it is convenient for us to think of our iid random variables as taking values in the space $\{0,1\}^\N$, rather than the more usual $[0,1]$.
Endow $\{0,1\}^\N$ with the product of the uniform measure on $\{0,1\}$.
Thus, each of our random variables is an infinite string of independent uniformly random bits.
Given $x \in (\{0,1\}^\N)^\Z=\{0,1\}^{\Z\times\N}$, a finite set $S \subset \Z\times \N$, a discrete space $A$ and a function $f:\{0,1\}^{\Z\times\N}\to A$, we say that \textbf{$f(x)$ is determined by the restriction of $x$ to $S$} if there exists an element $a\in A$ such that $f(x')=f(x)$ for almost every $x' \in \{0,1\}^{\Z\times\N}$ such that the resitrictions of $x$ and $x'$ to $S$ coincide.

We say that a factor $F:\{0,1\}^{\Z\times\N}\to A^\Z$ is \textbf{bit-finitary} if for almost every $x \in \{0,1\}^{\Z\times\N}$ (with respect to the product of the uniform measure on $\{0,1\}$), there exists a finite set $S \subset \Z\times\N$ such that $F(x)_0$ is determined by the restriction of $x$ to $S$.  For $i\in\m Z$ we let $r_i(x)$ be the minimal integer $r\geq 0$ such that $F(x)_i$ is determined by the restriction of $x$ to $\{i-r,\ldots,i+r\}\times\{0,\ldots,d\}$ for some $d<\infty$. We call $\bigl(r_i(x)\bigr)_{i\in \Z}$ the \textbf{bit-finitary coding radii} of $F$. We say that $F$ has finite expected bit-finitary coding radius if $\E r_0(U) <\infty$.

\begin{samepage}
\begin{thm} \label{thm:bit} Let $U= (U_{i,j})_{(i,j)\in\Z\times\N}$ be a collection of uniform $\{0,1\}$ random variables, let $A$ be a discrete space, and let $F:\{0,1\}^{\Z\times\N}\to A$ be a bit-finitary factor with finite expected bit-finitary coding radius.  Then there exists a countable set $\mathcal{K}$ and a factor
$ G : \{0,1\}^{\Z\times\N}\to \mathcal{K}^\Z$
such that the following claims hold.
\begin{enumerate}[(i)] \item The process  $G_i(U)$ is a Markov chain.
\item There exists a function $h:\mathcal{K}\to A$ such that
\[h \circ G_i(U)=F_i(X)\]
for all $i \in \Z$ almost surely.
\end{enumerate}
\end{thm}
\end{samepage}

It is not hard to see that the construction of the finitary factors in Theorem \ref{main} can be taken to be bit-finitary, so that it would be possible to deduce item (ii) of that theorem from Theorem \ref{thm:bit}. (For this example, however, there is a more obvious construction of the Markov chain, which we use to prove item (ii) of Theorem \ref{main} directly.)

\subsubsection*{Compact Markov chains.}
While we now have an entire family of finitely dependent colorings, all the known proofs of finite dependence rely on delicate cancellations in the finite dimensional distributions.  It is natural to seek examples whose finite dependence follows by more direct reasoning.  Here is one potential candidate.  Consider a stationary discrete-time Markov chain on a compact (but perhaps uncountable) metric space.  Any partition of the space into $q$ parts immediately gives a stationary $\{1,\ldots,q\}$-valued process.  Suppose that the Markov chain always moves by at least distance $\epsilon>0$ at every step.  Then by choosing a finite partition into parts of diameter less than $\epsilon$ (which is possible by compactness), we would obtain a stationary proper coloring.  Suppose that in addition the chain mixes perfectly $k$ steps, in the sense that from any initial state it is at stationarity at time $k$.  Then the coloring would be finitely dependent.

We do not know whether there exists a Markov chain with all the above properties. (In particular we do not know how to construct one by ``working backwards'' from the known colorings).  We show that there is no \emph{reversible} chain with the desired properties.
\begin{prop}\label{prop:noReversibleCompactChain}
  If $(X_n)_{n\in\m Z}$ is a stationary reversible Markov process on a compact metric space $(S,d)$ such that $X_0$ is independent of $X_k$ for some integer $k>0$, then there does not exist $\ve>0$ such that $d(X_0,X_1)\geq \ve$ almost surely.
\end{prop}

However, there is an example if one allows the state space to be non-compact.
\begin{prop}\label{prop:existsNonCompact}
  There exists a stationary Markov process $(X_n)_{n\in\m Z}$ on a (non-compact) countable metric space $(S,d)$ such that $X_0$ is independent of $X_{2}$ and $d(X_0,X_2)\geq 1$ a.s.
\end{prop}

\subsubsection*{Shifts of finite type and higher dimensions.} From a finitely dependent coloring of $\m Z$ one may construct finitely dependent colorings in higher dimensions, as well as finitely dependent processes satisfying more general local constraints (namely non-lattice shifts of finite type), as shown in \cite{HL}. Applying this to the colorings from Theorem~\ref{main} yields the next two corollaries, which we state after giving the necessary definitions.

The hypercubic lattice is the graph with vertex set $\m Z^d$ and an edge between $u$ and $v$ whenever $\|u-v\|_1=1$; the graph is also denoted $\m Z^d$. A process on $\m Z^d$ is stationary if it is invariant in law under all translations of $\m Z^d$, and it is ffiid if it is equal in law to $F(Y)$ where $Y$ is an iid process on $\m Z^d$ and $F$ is a translation-equivariant function satisfying the following property: for almost every sequence $y$ (with respect to the law of $Y$), there exists $r<\infty$ such that $F(y)_0=F(y')_0$ whenever $y'$ agrees with $y$ on $\{-r,\ldots,r\}^d$. Let $R(y)$ be the minimal such $r$. We call the random variable $R=R(Y)$ the {\bf coding radius} of the process.  A process indexed by a graph is {\bf $\bm{k}$-dependent} if its restrictions to two subsets of $V$ are independent whenever the subsets are at graph-distance greater than $k$ from each other.

The following is a consequence of our Theorem \ref{main} combined with methods of \cite{HL}.

\begin{cor}\label{cor:highDim}
  Let $d\geq 2$. There exist integers $q=q(d)$ and $k=k(d)$ such that:
  \begin{enumerate}
  \item there exists a ffiid $1$-dependent $q$-coloring of $\m Z^d$ with exponential tail on the coding radius;
  \item there exists a stationary ffiid $k$-dependent $4$-coloring of $\m Z^d$ with exponential tail on the coding radius.
  \end{enumerate}
\end{cor}

Coloring is a special case of the following more general notion, in which the requirement that adjacent colors differ is replaced with arbitrary local constraints. A {\bf shift of finite type} is a set of configurations $S$ characterized by an integer $k$ and a set $W\subseteq \{1,\ldots,q\}^k$ as follows:
$$
  S=S(q,k,W):=\bigl\{x\in \{1,\ldots,q\}^{\m Z}\colon (x_{i+1},\ldots,x_{i+k})\in W\ \forall i\in\m Z\bigr\}.
$$
We call the shift of finite-type {\bf non-lattice} if for some $w\in W$ we have that
$$
  \gcd\bigl\{t\geq 1\colon\exists x\in S\ s.t.\ (x_1,\ldots,x_k)=(x_{t+1},\ldots,x_{t+k})=w\bigr\}=1.
$$
\begin{cor}\label{cor:shifts}
  Let $S$ be a non-lattice shift of finite type on $\m Z$. There exists an integer $k$ (depending on $S$) and a $k$-dependent ffiid process $X$ with exponential tail on the coding radius such that the random sequence $X$ belongs to $S$ almost surely.
\end{cor}

Intermediate in generality between $q$-colorings and shifts of finite type is the class of stochastic processes taking values in the vertex set of a finite graph such that realizations of the process are a.s.\! paths. For the complete graph on $q$ vertices $K_q$, such a process is precisely a $q$-coloring. A natural modification of the construction of \cite{HL} was systematically investigated in \cite{levy2015}, in which the graph $K_q$ was replaced with a weighted graph. It was found that, other than straightforward modifications of the $1$-dependent $4$-coloring and $2$-dependent $3$-coloring of \cite{HL}, no other finitely dependent processes arise in this manner. It would be interesting to see if the obvious $t$-analogue of this result is true.

\subsubsection*{Outline of the paper}

Section~\ref{sec:background} covers background material and simple facts about the combinatorial objects we will use in the proof of the main theorem.
Section~\ref{sec:finiteMallowsColoring} constructs the colorings in the main theorem by starting on finite intervals and using Kolmogorov extension.
Section~\ref{sec:reverse} shows that the colorings are reversible.
Sections~\ref{sec:bubbles} and \ref{sec:coloringIntegers} complete the proof of the main theorem by providing a second construction of the colorings as a finitary factor with exponential tails on the coding radius.

The remaining results claimed in the introduction are proven in Sections~\ref{sec:bitproof} through \ref{sec:highDim}. Open problems are in Section \ref{sec:open}.

%%%%%%%%%%%%%%%%%%%%%%%%%%%
%%%%%%%%%%%%%%%%%%%%%%%%%%%
%%                       %%
%%       Section 2       %%
%%       Background      %%
%%                       %%
%%%%%%%%%%%%%%%%%%%%%%%%%%%
%%%%%%%%%%%%%%%%%%%%%%%%%%%

\section{Permutations, Codes, Colorings, and Graphs}
\label{sec:background}

This section introduces notation and basic facts used in the proof of Theorem~\ref{main}.
As stated previously, the essence of this theorem is that finitely dependent colorings arise as $t$-analogues of the random colorings in \cite{HL}.
The $t$-analogue of a positive integer $n$ is $[n]_t:=1+t+\cdots +t^{n-1}=(1-t^n)/(1-t)$.
Many numerical equalities that are combinatorial in nature generalize to polynomial identities between $t$-analogues.
This phenomenon appears frequently in algebraic combinatorics \cite{stanley1997enumerative}.
The $t$-\textbf{factorial} $[n]_t^!$ and the $t$-{\bf binomial coefficient} $\binom{n}{k}_t$ are defined via the formulas
$$
  [n]^!_t :=\prod_{k=1}^n[k]_t\quad\text{ and }\quad \binom{n}{k}_t:=\frac{[n]^!_t}{[k]^!_t[n-k]^!_t}.
$$
There are $n!$ permutations in $S_n$. A $t$-analogue of this fact is that
\begin{equation}\label{eq:tMalId}
  [n]^!_t=\sum_{\sigma\in S_n}t^{\inv(\sigma)},
\end{equation}
where the inversion number $\inv(\sigma)$ is defined to be
$$
  \inv(\sigma):=\#\bigl\{1\leq i<j\leq n\colon \sigma(i)>\sigma(j)\bigr\}.
$$
Equation \eqref{eq:tMalId} is well known and easy to prove, if one uses the right bijection (see e.g. \cite[Prop. 1.3.17]{stanley1997enumerative}). It also follows from the proof of Lemma~\ref{lem:geomMal} later in this section.

The \textbf{Mallows measure} $\Mal_t$ with parameter $t$ is the probability measure on $S_n$ assigning to each permutation $\sigma\in S_n$ a probability of $t^{\inv(\sigma)}/[n]_t^!$.
A {\bf Mallows random permutation} is a random element of $S_n$ whose law is a Mallows measure.
Since $\inv(\sigma)=\inv(\sigma^{-1})$, a Mallows random permutation is equal in law to its inverse.
Various statistics of Mallows random permutations are related to geometric random variables.
We will need a rather extensive array of variants of the geometric distribution.

Let $X$ be a random variable.
\begin{enumerate}
  \item $X$ is an \textbf{$i$-truncated, $t$-geometric} random variable if
  \[\P(X = j) = \frac{t^j}{1+t+\cdots+t^i},\qquad 0\leq j\leq i.\]
  \item $X$ is a \textbf{$u$-zero-weighted, $i$-truncated, $t$-geometric} random variable if
  \[\P(X = j) = \frac{u^{\mathbbm{1}[j=0]}t^j}{u+t+\cdots+t^i},\qquad 0\leq j\leq i.\]
  \item $X$ is a \textbf{$u$-max-weighted, $i$-truncated, $t$-geometric} random variable if
  $$
    \m P(X=j)=\frac{u^{\mathbbm{1}[j=n]}t^j}{1+t+\cdots+t^{n-1}+ut^n},\qquad 0\leq j\leq n.
  $$
  \item $X$ is a \textbf{$u$-end-weighted, $i$-truncated, $t$-geometric} random variable if
  \[
    \P(X = j) = \frac{u^{\mathbbm{1}[j\in\{0,i\}]}t^j}{u+t+\cdots+t^{i-1}+ut^i},\qquad 0\leq j\leq i.
  \]
  \item $X$ is a \textbf{$u$-zero-weighted, $t$-geometric} random variable if
  \[
    \P(X = j) = \frac{u^{\mathbbm{1}[j=0]}t^j}{u+\frac{t}{1-t}},\qquad 0\leq j<\infty.
  \]
\end{enumerate}

Intervals of integers are sets of the form $I\cap \m Z$ where $I$ is an interval of real numbers.
We write $\bbr{a,b}$ for $[a,b)\cap \m Z$ and we use similar blackboard-bold notation for other types of intervals as well.
The cardinality of a set is denoted by $\# S$.

As we will be constructing colorings directly on $\m Z$, it is convenient to introduce notation for permutations of arbitrary integer intervals $I$, which may be finite or infinite.
A permutation of $I$ is a bijection from $I$ to itself, and we write $\Sym(I)$ for the set of all such bijections.
We identify $\Sym(I)$ with the subset of $\Sym(\m Z)$ consisting of permutations fixing all elements of $\m Z\setminus I$.
A \textbf{finite permutation} is a permutation fixing all but finitely many integers.
For permutations $\sigma,\tau\in \Sym(I)$, we write $\sigma\circ \tau$ for the permutation mapping $i$ to $\sigma\bigl(\tau(i)\bigr)$.
For a sequence of permutations $\{\sigma_j\}_{j\in J}$ indexed by a finite interval $J=\bbb{a,b}$ of $\m Z$, we denote the composite permutation by
$$
\bigCircle_{j\in J}\sigma_j:=
  \sigma_{a}\circ \sigma_{a+1}\circ\cdots\circ \sigma_{b -1}\circ \sigma_{b}.
$$

The Lehmer code is a standard way of encoding permutations of finite intervals by a sequence of integers \cite{knuth1998art}.
We will use an extension of this to (possibly infinite) intervals $I$ of $\m Z$.
For such intervals, we define the \textbf{Lehmer code} to be the map $\sL\colon\Sym(I)\to\bbb{0,\infty}^I$ given by
$$
\sL(\sigma)_i=\#\bigl\{ j\in I\colon j>i\text{ and }\sigma(j)<\sigma(i)\bigr\},\quad \sigma\in\Sym(I),\ i\in I.
$$
The Lehmer code is a refinement of the \textbf{inversion number},
$$
\inv(\sigma):=\#\bigl\{(i,j)\in I^2\colon j>i\text{ and }\sigma(j)<\sigma(i)\bigr\},
$$
in the sense that $\sum_{i\in I}\sL(\sigma)_i=\inv(\sigma)$. A variant of the Lehmer code is the \textbf{insertion code}, which is the map $\sLtilde\colon \Sym(I)\to\bbb{0,\infty}^I$ given by
$$ \sLtilde(\sigma)_i:=\sL(\sigma)_{\sigma^{-1}(i)}=\#\bigl\{j\in I\colon j<i\text{ and }\sigma^{-1}(j)>\sigma^{-1}(i)\bigr\},\quad  \sigma\in\Sym(I),\ i\in I.
$$
The entries of $\sLtilde(\sigma)$ are a permutation of those in $\sL(\sigma)$, so $\sum_{i\in I}\sLtilde(\sigma)_i=\inv(\sigma)$.

\begin{figure}
\centering
\begin{subfigure}[t]{.49\textwidth}
\centering
{}\hspace*{1cm}
\begin{tikzpicture}[x=13pt,y=13pt]
  \draw[fill=blue!15!white] (7,0.5) rectangle (9.5,4);

  \foreach \p in {(1,6), (2,8), (3,7), (4,1), (5,9), (6,2), (7,4), (8,3), (9,5)}{
    \node at \p {\textbullet};
  }
  \node[right] at (10,1.5) {$\sL(\sigma)_i=1$};
  \draw[thick,->] (10,1.5) -- (8.5,1.5) node[midway,sloped,rotate=270] {};
  \draw[-] (0.5,0.5) -- (0.5,9.5) -- (9.5,9.5) -- (9.5,0.5) --cycle;
  \node[below] at (7,0.5) {$\vphantom{\sigma^{-1}(}i\vphantom{)}$};
  \node[right] at (9.5,4) {$\sigma(i)$};
\end{tikzpicture}
\caption{The Lehmer code.}
\label{sfig:L}
\end{subfigure}
\begin{subfigure}[t]{.49\textwidth}
\centering
{}\hspace*{1cm}
\begin{tikzpicture}[x=13pt,y=13pt]
  \draw[fill=blue!15!white] (3,0.5) rectangle (9.5,7);

  \foreach \p in {(1,6), (2,8), (3,7), (4,1), (5,9), (6,2), (7,4), (8,3), (9,5)}{
    \node at \p {\textbullet};
  }
  \node[right] at (10,1.5) {$\sLtilde(\sigma)_i=5$};
  \draw[thick,->] (10,1.5) -- (8.5,1.5) node[midway,sloped,rotate=270] {};
  \draw[-] (0.5,0.5) -- (0.5,9.5) -- (9.5,9.5) -- (9.5,0.5) --cycle;
  \node[below] at (3,0.5) {$\sigma^{-1}(i)$};
  \node[right] at (9.5,7) {$\vphantom{\sigma^{-1}(}i\vphantom{)}$};
\end{tikzpicture}
\caption{The insertion code.}
\label{sfig:Ltilde}
\end{subfigure}
\caption{The Lehmer code and the insertion code can be read from the scatter plot of a permutation. Visualization of (a) $\sL(\sigma)_7=1$ and (b) $\sLtilde(\sigma)_7=5$ for the permutation $\sigma=687192435$. The quantities are the numbers of dots in the shaded regions.}
\label{fig:scatter}
\end{figure}
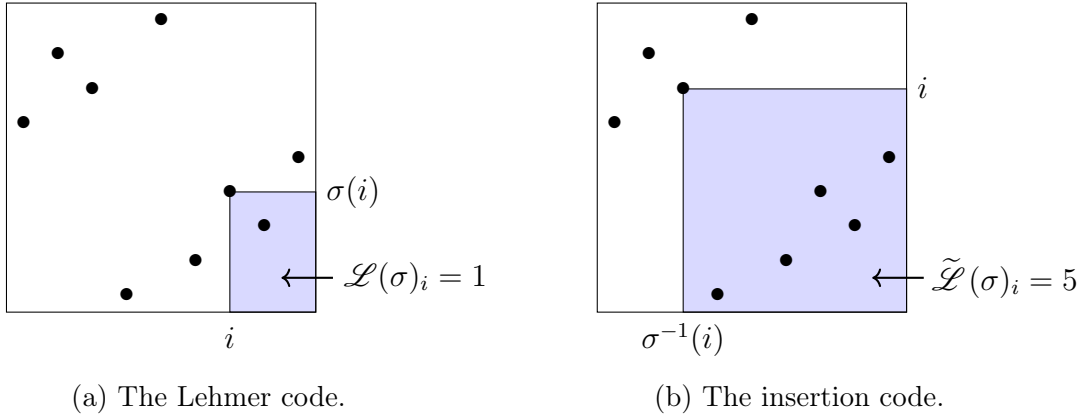

Clearly $0\leq \sL(\sigma)_i\leq \sup I-i$ and $0\leq \sLtilde(\sigma)_i\leq i-\inf I$ for all $i\in I$. Let
\begin{align*}
  \Omega_I&=\bigl\{\ell\in\bbb{0,\infty}^I\colon 0\leq \ell_i\leq \sup I-i,\ \forall i\in I\bigr\}\quad \text{and}\\
  \tilde\Omega_I&=\bigl\{\ell\in\bbb{0,\infty}^I\colon 0\leq \ell_i\leq i-\inf I,\ \forall i\in I\bigr\},
\end{align*}
so that $\sL\bigl(\Sym(I)\bigr)\subseteq \Omega_I$ and $\sLtilde\bigl(\Sym(I)\bigr)\subseteq \tilde\Omega_I$. When $I$ is finite, $\sL$ and $\sLtilde$ are bijections from $\Sym(I)$ to $\Omega_I$ and $\tilde \Omega_I$, respectively, with explicit inverse functions which we now describe.

For a finite interval $J$ of $\m Z$, let $\pi^-_J\in\Sym(\m Z)$ denote the permutation fixing $\m Z\setminus J$ and cyclically decrementing $J$, i.e.,
$$
  \pi_J^-(j)=
  \begin{cases}
    j-1,    & \min J<j\leq \max J\\
    \max J, & j=\min J\\
    j,      & j\in \m Z\setminus J.
  \end{cases}
$$
The permutation $\pi_J^-$ has a single cycle and it is of size $\# J$.

\begin{figure}[t]
    \centering
    \includegraphics[width=0.3\textwidth]{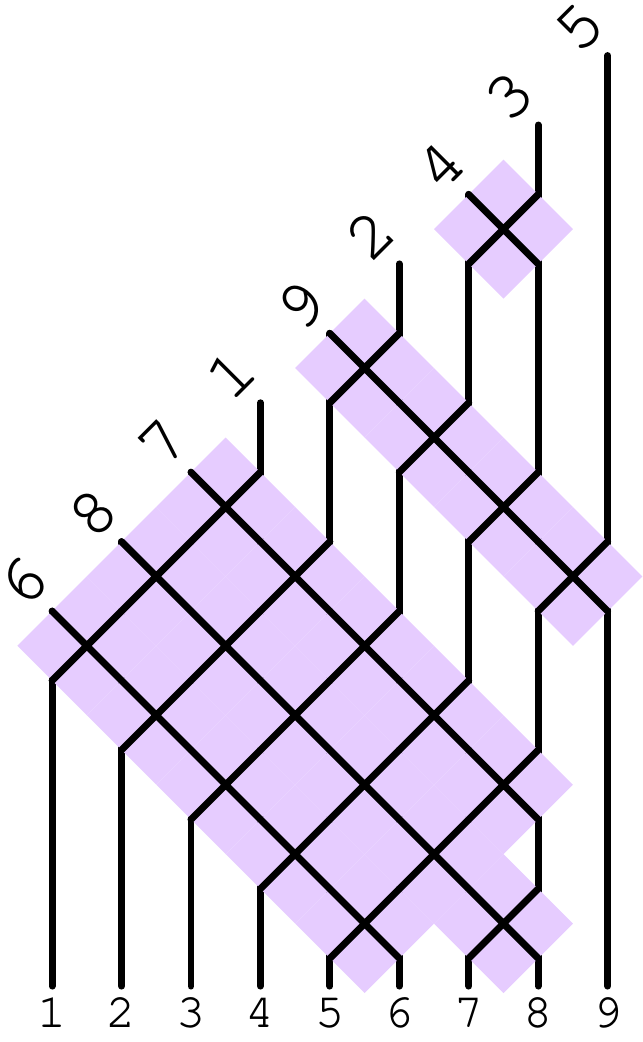}
    \caption{
      Depiction of $\sD(\ell)$ for the sequence $\ell=(5,6,5,0,4,0,1,0,0)$.
      Given $\ell$, draw crosses on NW-SE diagonals such that the $i^{th}$ such diagonal from the left contains $\ell_i$ crosses.
      Then follow the wires upwards to obtain the permutation $\sigma=\sD(\ell)$.
      In this case $\sigma=687192435$.
    }
    \label{fig:wire}
  \end{figure}

Let $\sD\colon\bbr{0,\infty}^I\to \Sym(\m Z)$ denote the map
\begin{equation}\label{eq:D}
  \sD(\ell)=\bigCircle_{i\in I}{}\pi^-_{\bbb{i,i+\ell_i}},\qquad \ell\in \bbr{0,\infty}^I.
\end{equation}
See Figure~\ref{fig:wire}. Also let $\sDtilde\colon\bbr{0,\infty}^I\to \Sym(\m Z)$ denote the map
\begin{equation}\label{eq:Dtilde}
  \sDtilde(\ell)=\bigCircle_{i\in I}{}\pi^-_{\bbb{i-\ell_i,i}},\qquad \ell\in \bbr{0,\infty}^I.
\end{equation}
\begin{lemma}\label{lem:LDinv}
  Suppose that $I$ is a finite interval of $\m Z$.
  Then $\sL$ is a bijection from $\Sym(I)$ to $\Omega_I$, with inverse given by the restriction of $\sD$ to $\Omega_I$.
  Similarly $\sLtilde$ is a bijection from $\Sym(I)$ to $\tilde\Omega_I$, with inverse given by the restriction of $\sDtilde$ to $\tilde \Omega_I$.
\end{lemma}
The proof, which is straightforward, is omitted. Similar results are well known and appear for example in
\cite{MR3334280,gnedin2012two,stanley1997enumerative}.
\begin{lemma}\label{lem:geomMal}
For $1\leq i\leq n$ let $G_i$ be an $(n-i)$-truncated $t$-geometric random variable,
and suppose that $G_1,\ldots,G_{n}$ are independent.
Then the law of $\sD(G_1,\ldots,G_{n})$ is the Mallows measure on $S_n$ with parameter $t$.
\end{lemma}
\begin{proof} It follows from Lemma~\ref{lem:LDinv} that for all $\sigma\in S_n$, $$\m P\bigl(\sD(G_1,\ldots,G_{n})=\sigma\bigr)=
\m P\bigl((G_1,\ldots,G_{n})=\sL(\sigma)\bigr).$$
By the independence of $G_1,\ldots,G_{n}$, the right side equals
$$\prod_{i=1}^{n}\m P\bigl(G_i=\sL(\sigma)_{i}\bigr).$$
The result now follows since $\sum_{i=1}^n\sL(\sigma)_i=\inv(\sigma)$.
\end{proof}

  Having discussed permutations and their encodings, we now relate these objects to words.
  Fix a (possibly infinite) interval $I$ of $\m Z$ and an integer $q\geq 1$.
  A \textbf{word} $x=(x_i)_{i\in I}$ indexed by $I$ is a function from $I$ to $\bbb{1,q}$, and its entries are referred to as \textbf{characters}.
  It is \textbf{proper} if $x_i\not=x_{i+1}$ whenever $i,i+1\in I$. A \textbf{$\bm{q}$-coloring} is defined to be a proper word.
  The \textbf{length} of a word, denoted by $|x|$, is the cardinality of its index set $I$. We denote the empty word by $\emptyset$.
  The \textbf{concatenation} of a word $x$ indexed by $\bbb{a,b}$ with a word $y$ indexed by $\bbl{b,c}$ is the word $xy$ indexed by $\bbb{a,c}$ whose restrictions to $\bbb{a,b}$ and $\bbl{b,c}$ are $x$ and $y$, respectively.
  Similar notation is used for concatenations of words with individual characters.
  Given a word $x$ indexed by $I$ and a set $A\subseteq I$ of size $m$, the \textbf{subword} $(x_i\colon i\in A)$ is defined to be the word $x_{i_1}x_{i_2}\cdots x_{i_m}$, where $i_k$ is the $k^{th}$ smallest element of $A$.
  In other words, it is the subsequence of $x$ indexed by $A$.

  Following \cite{HL}, we say that a permutation $\sigma$ of $I$ is a \textbf{proper building} of a word $x$ if for each $t\in I$ the subword
  $$
    x^\sigma(t):=\bigl(x_i\colon \sigma(i)\leq t\bigr)
  $$
  of $x$ is proper.
  We write $\sigma\vdash x$ if this occurs.
  Note that $x$ is proper if and only if it has some proper building, in which case for instance the identity permutation is a proper building.

  The following picture will be very useful. We regard $\sigma(i)$ as the arrival time of $i$. Then $x^\sigma(t)$ is the subword of $x$ that has arrived by time $t$. At time step $t$, the integer $\sigma^{-1}(t)$ arrives, and the character $x_{\sigma^{-1}(t)}$ is inserted into $x^\sigma(t-1)$ (or the empty word, if $t=\min I$). The insertion code $\sLtilde(\sigma)_t$ has a natural interpretation in terms of arrivals, from which its name derives: it is the distance from the right at which $x_{\sigma^{-1}(t)}$ is inserted in $x^\sigma(t-1)$. More precisely, for all $t>\min I$ there are subwords $u$ and $v$ of $x$ such that
  $$
    x^\sigma(t-1)=uv,\qquad x^\sigma(t)=ux_{\sigma^{-1}(t)}v,\quad\text{and}\quad |v|=\sLtilde(\sigma)_t.
  $$
  For example if $\sigma=25431$ and $t=4$, then
  $$
    x^\sigma(t-1)=x_1x_4x_5,\qquad x_{\sigma^{-1}(t)}=x_3,\qquad x^\sigma(t)=x_1x_3x_4x_5,
  $$
  and $\sLtilde(\sigma)_t=|x_4x_5|=2$.

  The condition $\sigma\vdash x$ can also be expressed in the language of graphs.
  Let $I$ be the index set of $x$. Then for a graph $G$ with vertex set $I$, we say that $x$ is a \textbf{proper coloring} of $G$ if $x_i\not=x_j$ whenever $i$ and $j$ are adjacent in $G$.
  Given a permutation $\sigma$ of $I$, we define its \textbf{constraint graph} $\Gamma_\sigma$ to have vertex set $I$ and an undirected edge between $i$ and $j$, where $i<j$, if and only if $\sigma(i)<\sigma(k)>\sigma(j)\text{ for all } k\in\bbo{i,j}.$
  In other words, two integers are adjacent in $\Gamma_\sigma$ if and only if they both arrive prior to any integer between them.
  It is immediate from the definitions that $\sigma\vdash x$ if and only if $x$ is a proper coloring of $\Gamma_\sigma$.
  See Figure~\ref{fig:constraint} for an example of a constraint graph.

\begin{figure}
\centering
\includegraphics[width=.6\textwidth]{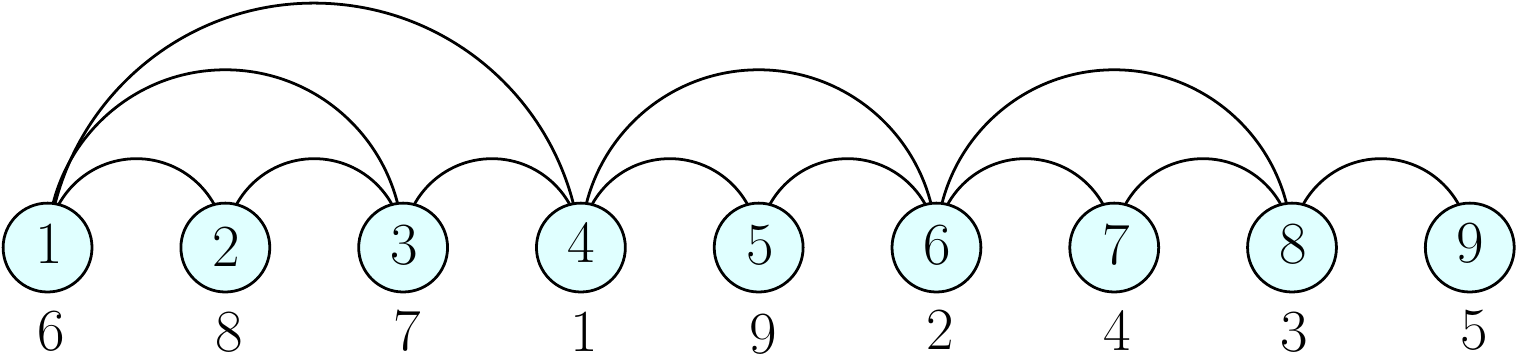}
\caption{Constraint graph of the permutation $\sigma=687192435$.
Vertices of the constraint graph are labeled $1,\ldots,9$.
The image under $\sigma$ is written below each vertex, and these are interpreted as arrival times.
An arc is drawn between two vertices if they arrive before every vertex between them.
}
\label{fig:constraint}
\end{figure}

  Next we consider decompositions of graphs that arise naturally in the context of buildings.
  Let $G$ be a graph whose vertex set $I$ is a (possibly infinite) interval of $\m Z$.
  An integer $i\in I$ is a \textbf{bubble endpoint} of $G$ if there do not exist $j$ and $k$ with $j<i<k$ such that $j$ and $k$ are adjacent in $G$.
  A \textbf{bubble} of $G$ is a subgraph induced by a finite interval of $\m Z$ whose endpoints are consecutive bubble endpoints, and $\bub(G)$ denotes the set of all bubbles.
  Note that every endpoint of $I$ is a bubble endpoint.
  \begin{lemma}\label{lem:constraintGraphDecomp}
    Let $G$ be a graph whose vertex set $I$ is a (possibly infinite) interval  of $\m Z$.
    Then $G=\bigcup \bub(G)$ iff the infimum and supremum of the set of bubble endpoints agree with those of $I$.
  \end{lemma}
  The proof is immediate. Note that the condition on the set of bubble endpoints holds automatically if $I$ is finite.

  Next we consider bubbles of constraint graphs of permutations. An integer $i$ is a \textbf{record} of a permutation $\sigma$ if it is either the maximum or minimum of the set
  $$
    \bigl\{\sigma(j)\colon j\leq \sigma^{-1}(i)\bigr\}.
  $$
  Records are a well-studied permutation statistic, both combinatorially \cite{bona2012combinatorics} and probabilistically  \cite{kortchemski2009asymptotic,glick1978breaking,chern2000distribution,auger2016analysis}. A \textbf{founder} of a permutation $\sigma$ is defined to be a record of $\sigma^{-1}$. Equivalently, $i$ is a founder of $\sigma$ iff there do not exist $j$ and $k$ with $j<i<k$ and $\sigma(j)<\sigma(i)>\sigma(k)$. Phrased in terms of the arrival times picture, $i$ is a founder if and only if either:
  \begin{itemize}
    \item it arrives prior to all smaller elements of $I$, or
    \item it arrives prior to all larger elements of $I$.
  \end{itemize}
  We write $\c F(\sigma)$ for the set of founders of $\sigma$.
  \begin{lemma}\label{lem:fRecL}
    Let $\sigma$ be a permutation of an interval $I\subseteq \m Z$ (which may be finite or infinite). Then the set of bubble endpoints of the constraint graph $\Gamma_\sigma$ is $\c F(\sigma)$. Furthermore if $I=\bbb{0,n}$, then $\#\c F(\sigma)=\#\bigl\{i\in I\colon \sLtilde(\sigma)_i\in\{0,i\}\bigr\}$.
  \end{lemma}
  In particular, this implies that $\#\!\bub(\Gamma_\sigma)$ is one less than the number of founders of $\sigma$ (which could be infinite for a permutation of an infinite interval). Notation involving Lehmer codes often simplifies when working on the interval $\bbb{0,n}$. Note that this interval has $n+1$ elements.

  \begin{proof}[Proof of Lemma~\ref{lem:fRecL}]
    If $i$ is not a bubble endpoint of $\Gamma_\sigma$, then there must exist $j<i<k$ such that $j$ and $k$ arrive prior to all elements of $\bbo{j,k}$, and in particular $i$. Thus $i$ is not a founder. Conversely if $i$ is not a founder, then there exist $j<i<k$ with $\sigma(j)<\sigma(i)>\sigma(k)$. Now choose $j$ maximal and $k$ minimal satisfying these conditions to obtain an edge of $\Gamma_\sigma$ passing over $i$, showing that $i$ is not a bubble endpoint. This establishes the first claim.

    For the second claim, observe that
    $$\c F(\sigma)=\bigl\{i\in I\colon\sLtilde(\sigma)_{\sigma(i)}\in\{0,\sigma(i)\}\bigr\}=\sigma\bigl(\bigl\{i\in I\colon\sLtilde(\sigma)_{i}\in\{0,i\}\bigr\}\bigr).$$
    Thus $\#\c F(\sigma)=\#\bigl\{i\in I\colon\sLtilde(\sigma)_{i}\in\{0,i\}\bigr\}$.
  \end{proof}

  Let $\Col_q(G)$ be the number of proper $q$-colorings of a graph $G$.
  \begin{lemma}\label{lem:colConst}
    For any $q\geq 3$, any finite interval $I$ of $\m Z$, and any permutation $\sigma$ of $I$,
    \begin{equation*}\label{eqn:wordCount}
      \Col_q\bigl(\graph_\sigma\bigr)= q(q-2)^{\# I-1}\Bigl(\frac{q-1}{q-2}\Bigr)^{\#\!\bub(\Gamma_\sigma)}.
    \end{equation*}
  \end{lemma}
  \begin{proof}
    Without loss of generality assume that $I=\bbb{0,n}$. For each $i\in\bbb{0,n}$, let $\graph_\sigma^{i}$ denote the subgraph of $\graph_{\sigma}$ induced by the set of the first $i$ vertices to arrive, $\sigma^{-1}(\bbb{0,i})$. The graph $\graph_\sigma^{1}$ can be colored in $q$ ways. Suppose $x$ is a proper coloring of $\graph_{\sigma}^{i}$. If $i+1$ is a record of $\sigma^{-1}$, then $i+1$ has degree $1$ in $\graph_{\sigma}^{i+1}$, and there are exactly $q-1$ colorings of $\graph_\sigma^{i+1}$ extending $x$. Otherwise, $i+1$ has degree $2$ in $\graph_{\sigma}^{i+1}$ and, since both of the neighbors of $i+1$ in $\graph_\sigma^{i}$ have different colors in $x$, there are exactly $q-2$ proper colorings of $\graph_\sigma^{i+1}$ extending $x$. Thus by Lemma~\ref{lem:fRecL},
    \begin{equation*}\label{eq:colorCases}
      \Col_q\bigl(\graph_\sigma^{i+1}\bigr) =
        \begin{cases}
          (q-1)\Col_q\bigl(\graph_\sigma^{i}\bigr),& \text{if }i+1 \text{ is a record of }\sigma^{-1}\\
          (q-2)\Col_q\bigl(\graph_\sigma^{i}\bigr),&  \text{otherwise.}
        \end{cases}
    \end{equation*}
    Since the records of $\sigma^{-1}$ are the bubble endpoints of $\bub(\Gamma_\sigma)$, the lemma follows.
  \end{proof}

  A key element of the proof of our main theorem is a joint probability measure on colorings and permutations of a finite interval.
  The marginal law of the coloring will provide the finite-dimensional distributions for our coloring of $\m Z$.
  The marginal law of the permutation will belong to the following family of permutation measures.
  Let $I\subset \m Z$ be a finite interval and let $u$ and $t$ be non-negative real parameters.
  The \textbf{bubble-biased Mallows measure} with parameters $t$ and $u$ is the probability measure $\BMal=\BMal_{t,u}=\BMal_{t,u}^I$ on $\Sym(I)$ given by
  \begin{equation}\label{eq:bmalDef}
    \BMal(\{\sigma\})=
      \frac{u^{\#\!\bub(\Gamma_\sigma)}t^{\inv(\sigma)}}
      {\sum_{\tau\in\Sym(I)}u^{\#\!\bub(\Gamma_\tau)}t^{\inv(\tau)}},\qquad \sigma\in\Sym(I).
  \end{equation}
  Since $\#\!\bub(\Gamma_\sigma)$ is one less than the number of records of $\sigma^{-1}$, we also have that
  $$
    \BMal(\{\sigma\})=
      \frac{u^{\#\{\text{records of }\sigma^{-1}\}}t^{\inv(\sigma)}}
      {\sum_{\tau\in\Sym(I)}u^{\#\{\text{records of }\tau^{-1}\}}t^{\inv(\tau)}},\qquad \sigma\in\Sym(I).
  $$
  The key property of $\BMal$ is that, for $q\geq 3$ and $u=\frac{q-1}{q-2}$, we have that
  \begin{equation}\label{eq:bMalCol}
    \BMal_{t,u}(\{\sigma\})=\frac{\Col_q(\Gamma_\sigma)t^{\inv(\sigma)}}{\sum_{\tau\in\Sym(I)}\Col_q(\Gamma_\tau)t^{\inv(\tau)}},\qquad \sigma\in\Sym(I),
  \end{equation}
  which follows by combining Lemma~\ref{lem:colConst} with \eqref{eq:bmalDef}.
  That is, $\BMal$ is the law of a Mallows random permutation biased by the number of proper $q$-colorings of its constraint graph. We extend the definition of the Mallows measure from permutations of $\bbb{1,n}$ to permutations of an arbitrary finite interval $I$ by declaring it to be the special case $u=1$ of $\BMal_{t,u}$. We remark that, even though $\Mal_t(\sigma)=\Mal_t(\sigma^{-1})$ for all $\sigma$, the quantities $\BMal_{t,u}(\sigma)$ and $\BMal_{t,u}(\sigma^{-1})$ differ in general when $u\not=1$. This fact adds significant complications to our proof of the main theorem.

  \begin{lemma}\label{lem:LehmerBub}
    For all $i\in\bbb{0,n}$, let $G_i$ be a $u$-end-weighted, $i$-truncated, $t$-geometric random variable, and suppose that $G_0,\ldots,G_{n}$ are independent.
    Then the law of $\sDtilde(G_0,\ldots,G_{n})$ is $\BMal_{t,u}^{\bbb{0,n}}$.
  \end{lemma}
  \begin{proof}
    It follows from Lemma~\ref{lem:LDinv} that for all $\sigma\in \Sym(\bbb{0,n})$,
    $$
      \m P\bigl(\sDtilde(G_0,\ldots,G_{n})=\sigma\bigr)=\m P\bigl((G_0,\ldots,G_{n})=\sLtilde(\sigma)\bigr).
    $$
    By the independence of $G_0,\ldots,G_{n}$, the right side equals
    $$
      \prod_{i=0}^{n}\m P\bigl(G_i=\sLtilde(\sigma)_{i}\bigr)=\prod_{i=0}^{n}\frac{u^{\mathbbm{1}[\sLtilde(\sigma)_i\in\{0,i\}]}t^{\sLtilde(\sigma)_i}}{u+t+\cdots+ut^{i-1}+ut^i}.
    $$
    The exponent of $u$ in this product is $\#\c F(\sigma)$ by Lemma~\ref{lem:fRecL}, and the exponent of $t$ is $\inv(\sigma)$. Thus
    $$
      \m P\bigl(\sDtilde(G_0,\ldots,G_{n})=\sigma\bigr)=\frac{u^{\#\c F(\sigma)}t^{\inv(\sigma)}}{\prod_{i=0}^n (u+t+\cdots+ut^{i-1}+ut^i)}.
    $$
    The result now follows since $\#\c F(\sigma)=\#\!\bub(\Gamma_\sigma)+1$.
  \end{proof}

%%%%%%%%%%%%%%%%%%%%%%%%%%%
%%%%%%%%%%%%%%%%%%%%%%%%%%%
%%                       %%
%%       Section 3       %%
%%       Finite dep.     %%
%%                       %%
%%%%%%%%%%%%%%%%%%%%%%%%%%%
%%%%%%%%%%%%%%%%%%%%%%%%%%%

\section{Finite dependence}\label{sec:finiteMallowsColoring}

The purpose of this section is to construct random colorings of $\m Z$ that are finitely dependent by starting on finite intervals and appealing to Kolmogorov extension.
\begin{prop}\label{prop:sec3main}
  Fix $t\in [0,1]$ and $q\geq 3$. Let $u=\frac{q-1}{q-2}$. Then there is a random $q$-coloring of $\m Z$ such that for every finite interval $I$, its restriction to $I$ has the law of a uniform $q$-coloring of the constraint graph of a $\BMal_{t,u}$-distributed permutation of $I$.

  The random coloring is strictly $k$-dependent iff $(q,k,t)$ satisfies the tuning equation
  \begin{equation}\label{eq:qktAna}
    qt[k]_t=[2]_t[k+1]_t.
  \end{equation}
\end{prop}
Note that the equation \eqref{eq:qktAna} is equivalent to \eqref{eqn:qkt} provided $t\not=1$.

Consider the probability measure $\Joint=\Joint_{t,q,n}$ on  $S_n\times \bbb{1,q}^n$ given by
$$
\Joint\Bigl(\bigl\{(\sigma,x)\bigr\}\Bigr)=\frac{\mathbbm{1}[\sigma\vdash x]t^{\inv(\sigma)}}{Z(t,q,n)},
$$
where the normalizing constant $Z(t,q,n)$ is
$$
Z(t,q,n)=\sum_{\sigma\in S_n}\sum_{x\in\bbb{1,q}^n}\mathbbm{1}[\sigma\vdash x]t^{\inv(\sigma)}=\sum_{\sigma\in S_n}\Col_q(\sigma)t^{\inv(\sigma)}.
$$

The permutation marginal of $\Joint$ is $\BMal$, by \eqref{eq:bMalCol}.
If the random pair $(\sigma,x)$ has law $\Joint$, then the conditional law of $x$ given $\sigma$ is the uniform measure on proper $q$-colorings of $\Gamma_\sigma$. We denote the marginal probability mass function of $x$ by $P^{\col}=P^{\col}_{t,q,n}$,
\begin{equation}\label{eq:pcol}
  P^{\col}(x)=\sum_{\sigma\in S_n}\frac{\mathbbm{1}[\sigma\vdash x]t^{\inv(\sigma)}}{Z(t,q,n)}.
\end{equation}
\begin{lemma}\label{lem:kolm}
  Let $t\in[0,1]$ and $q\geq 3$. There exists a measure $\MalCol=\MalCol_{q,t}$ on $\bbb{1,q}^{\m Z}$ such that if $X=(X_i)_{i\in\m Z}$ is random with law $\MalCol$, then $X$ is stationary and
  \begin{equation*}\label{eq:kolmo}
    \m P\bigl[(X_{i+1},\ldots,X_{i+n})=x\bigr]=\ptcoln{n}(x)
  \end{equation*}
  for all $i\in\m Z$, for all $n\geq 0$, and for all $x\in\bbb{1,q}^n$.
\end{lemma}
The random colorings of $\m Z$ which we construct in the proof of the main theorem have law $\MalCol_{q,t}$ for certain values of $t$.

By inspection of \eqref{eq:pcol}, observe that if $\sigma$ is a Mallows random permutation then
\begin{equation}\label{eq:malPcol}
  P^{\col}(x)=\frac{[n]_t^!}{Z(t,q,n)}\m P(\sigma\vdash x),\qquad x\in\bbb{1,q}^n.
\end{equation}
This characterization of $P^{\col}$ will be used in Section~\ref{sec:reverse} to prove reversibility.

The \textbf{building number} $B_t(x)$ is the unnormalized version of $\ptcoln{n}$ given by
\begin{equation*}\label{eqn:buildingRatio}
  B_t(x):=\sum_{\sigma\in S_n}\mathbbm{1}[\sigma\vdash x]t^{\inv(\sigma)}=Z(t,q,n)P^{\col}(x),\qquad x\in\bbb{1,q}^n.
\end{equation*}
This specializes when $t=1$ to the number of proper buildings, which was a key player in the earlier construction of \cite{HL}. Observe that
\begin{equation*}
  \sum_{x\in\bbb{1,q}^n}B_t(x)=\ztqn{n}.
\end{equation*}
and that $B_t(\emptyset)=1$.

The reason we use $B_t(x)$ (rather than using $P^{\col}$ directly) is that it satisfies simpler recurrences, as we will now see. We abbreviate a word $x=(x_i)_{i\in\bbb{1,n}}$ by writing $x=x_1\cdots x_n$, and we use the notation $\hat x_i:=x_1\cdots x_{i-1}x_{i+1}\cdots x_n$.
  \begin{lemma}\label{lem:recurrenceForProb}
    For all $n\geq 1$, all words $x \in \bbb{1,q}^n$, and all real $t\geq 0$ we have
    \begin{equation}\label{eq:Brecurrence}
      \Bn(x) = \mathbbm{1}[x \text{\emph{ is proper}}]\sum_{i=1}^n t^{n-i}\Bn(\hat x_i),\vspace{-5mm}
    \end{equation}
    and
    \begin{equation}\label{eq:Brecurrence2}
      \Bn(x)=\sum_{i=1}^{n}t^{n-i}\Bn(\hat x_i)-[2]_t\sum_{j=2}^{n}\mathbbm{1}[x_{j-1}=x_{j}]t^{n-j}\Bn(\hat x_j).
    \end{equation}
  \end{lemma}

  Equation \eqref{eq:Brecurrence} is a $t$-analogue of \cite[Prop. 9]{HL}.
  The variant recurrence \eqref{eq:Brecurrence2} (which was not used in \cite{HL}) simplies a large amount of casework.
  As an alternative to the proof below, one may deduce \eqref{eq:Brecurrence2} from \eqref{eq:Brecurrence} via the M\"obius Inversion Formula for posets \cite[Section 3.7]{stanley1997enumerative}.

  \begin{proof}[Proof of Lemma~\ref{lem:recurrenceForProb}]
    To prove equation \eqref{eq:Brecurrence}, observe that the permutation $\sigma$ is a proper building of $x$ with $\sigma^{-1}(n)=i$ if and only if $x$ is proper and the permutation $\hat \sigma_i:=\sigma_1\cdots\sigma_{i-1}\sigma_{i+1}\cdots\sigma_n\in S_{n-1}$ is a proper building of $\hat x_i$. Now \eqref{eq:Brecurrence} follows from the easy observation that $\inv(\sigma)=\inv(\hat \sigma_i)+n-i$.

    To establish \eqref{eq:Brecurrence2}, write ${\mathbbm{1}[\text{$x$ is proper}]}$ as ${\mathbbm{1}[x_1\not=x_2]\cdots\mathbbm{1}[x_{n-1}\not=x_n]}$. Then by \eqref{eq:Brecurrence},
    \begin{equation}\label{eqn:gettingMobius}
      \Bn(x)=\prod_{j=2}^{n}\bigl(1-\mathbbm{1}[x_{j-1}=x_{j}]\bigr)\sum_{i=1}^{n}t^{n-i}\Bn(\hat x_i).
    \end{equation}
    Next observe that for any pair of distinct indices $i\not=j$, the expression $$\mathbbm{1}[x_{j-1}=x_j]\,\mathbbm{1}[x_{i-1}=x_i]\,\Bn(\hat x_k)$$
    vanishes for all $k$. Indeed, any word $x$ with $x_{j-1}=x_j$ and $x_{i-1}=x_i$ must still have adjacent repeated indices even after deleting an arbitrary symbol, and so the resulting word has no proper buildings. Expanding \eqref{eqn:gettingMobius} and discarding such terms,
  \begin{align*}
    \Bn(x)&=\Bigl(1-\sum_{j=2}^{n}\mathbbm{1}[x_{j-1}=x_{j}]\Bigr)\sum_{i=1}^{n}t^{n-i}\Bn(\hat x_i).
  \end{align*}
  If the expression $\mathbbm{1}[x_{j-1}=x_j]\Bn(\hat x_k)$ is non-zero, then $k\in\{j-1,j\}$ and $\hat x_k=\hat x_j$. Thus
  \begin{align*}
    \Bn(x)&=\sum_{i=1}^{n}t^{n-i}\Bn(\hat x_i)-\sum_{j=2}^{n}\mathbbm{1}[x_{j-1}=x_{j}]\sum_{i=1}^{n}t^{n-i}\Bn(\hat x_i)\\
    &=\sum_{i=1}^{n}t^{n-i}\Bn(\hat x_i)-\sum_{j=2}^{n}\mathbbm{1}[x_{j-1}=x_{j}](t^{n-j}+t^{n-j+1})\Bn(\hat x_j).
  \end{align*}
  Factoring out $[2]_t=t+1$ from the second term in the latter expression yields \eqref{eq:Brecurrence2}.
\end{proof}

Using these recurrences, we show that the marginals of $\MalCol$ are consistent.
\begin{prop}[Consistency]\label{lem:consistency}
  For all $x\in \bbb{1,q}^n$ and all $a\in \bbb{1,q}$ we have that
  $$
    \sum_{a\in\bbb{1,q}}\ptcoln{n+1}(ax) = \sum_{a\in\bbb{1,q}}\ptcoln{n+1}(xa)=\ptcoln{n}(x).
  $$
\end{prop}

In the following proof and for the remainder of this section, we write $\star$ to denote a dummy variable that is summed over $\bbb{1,q}$. For example, given a function $f\colon \bbb{1,q}^k \to \R$ and a word $x$, we write $f(x\star^k)$ as a shorthand for $\sum_{y \in \bbb{1,q}^k} f(xy)$.

\begin{proof}
  To prove that $\ptcoln{n+1}(\star x)=\ptcoln{n}(x)$, we establish by induction on $n$ that
  \begin{equation}\label{eqn:Bconsist}
    \Bn(x \star)=(q[n+1]_t-[2]_t[n]_t)\Bn(x),\qquad x\in\bbb{1,q}^n.
  \end{equation}
  This is clear when $n=0$. Let $x$ be a word of length $n$, and suppose that \eqref{eqn:Bconsist} holds for all words of length at most $n-1$.  Also suppose that $x$ is proper, for otherwise \eqref{eqn:Bconsist} is trivial.

  Applying equation \eqref{eq:Brecurrence2} from Lemma \ref{lem:recurrenceForProb}, we see that for any $a\in \bbb{1,q}$ we have that
  \begin{equation}
    \Bn(xa)=t\sum_{i=1}^n t^{n-i}\Bn(\hat x_ia)+\Bn(x)-\mathbbm{1}[x_n=a][2]_t \Bn(x).
  \end{equation}
  Summing over all $a\in \bbb{1,q}$ yields that
  \begin{align*}
    \Bn(x \star)&=t\sum_{i=1}^n t^{n-i}\Bn(\hat x_i  \star)+q\Bn(x)-[2]_t \Bn(x).
  \end{align*}
  Hence by the inductive hypothesis
  \begin{align}\label{eq:lastStepConsist}
    \Bn(x \star)&=t(q[n]_t-[2]_t[n-1]_t)\sum_{i=1}^n t^{n-i}\Bn(\hat x_i)+q\Bn(x)-[2]_t \Bn(x).
  \end{align}
  By \eqref{eq:Brecurrence} we have that $\sum_{i=1}^n t^{n-i}\Bn(\hat x_i)=\Bn(x)$. Substituting this into \eqref{eq:lastStepConsist} and using the trivial identity $t[n]_t+1=[n+1]_t$, we deduce \eqref{eqn:Bconsist}.

  Combining \eqref{eqn:Bconsist} with \eqref{eqn:buildingRatio} yields that $\ptcoln{n+1}(x\star)=\ptcoln{n}(x)$. By an analogous argument, we have that ${\ptcoln{n+1}(\star x)=\ptcoln{n}(x)}$ as well.
\end{proof}
Lemma~\ref{lem:kolm} now follows easily.
\begin{proof}[Proof of Lemma~\ref{lem:kolm}]
  The family of cylinder measures in the statement of the lemma is consistent, by Proposition \ref{lem:consistency}. Thus by the Kolmogorov extension theorem \cite{kallenberg2006foundations} there exists a random coloring, $X$, for which \eqref{eq:kolmo} holds. Stationarity of $X$ is immediate.
\end{proof}

Next we derive an expression for the normalizing constant $\ztqn{n}$. 
When $(q,k,t)$ satisfies the tuning equation \eqref{eq:qktAna}, extra simplifications occur, as can be seen already in the following lemma.
\begin{lemma}\label{lem:zkqt}
  For all integers $n\geq 1$ we have that
  \begin{equation}\label{eqn:znqPartition}
    \ztqn{n}=\prod_{j=1}^{n}\bigl(q[j]_t-[2]_t[j-1]_t\bigr).
  \end{equation}
  Moreover, when $qt[k]_t=[2]_t[k+1]_t$, equation \eqref{eqn:znqPartition} may be rewritten as
  \begin{equation}\label{eqn:zkqPartition}
    \ztqn{n}=[n]^!_t\left(\frac{q}{[k+1]_t}\right)^n\binom{k+n}{k}_t.
  \end{equation}
\end{lemma}

\begin{proof}
  It follows from \eqref{eqn:Bconsist} that if $x$ is a word of length $n$ then $B_t(x\ \star)=\bigl(q[n+1]_t-[2]_t[n]_t\bigr)B_t(x)$. Summing over $x\in \bbb{1,q}^n$ yields that
  $$
    \ztqn{n+1}=\bigl(q[n+1]_t-[2]_t[n]_t)\ztqn{n}.
  $$
  Now a simple induction establishes \eqref{eqn:znqPartition}.

  Next, suppose that $(q,k,t)$ satisfies the tuning equation \eqref{eq:qktAna}. We show that for every $j\geq 1$,
  \begin{equation}\label{eqn:qjt}
    \bigl(q[j]_t-[2]_t[j-1]_t\bigr)[k+1]_t=q[k+j]_t.
  \end{equation}
  Indeed, the tuning equation allows us to substitute $qt[k]_t$ in place of $[2]_t[k+1]_t$. Furthermore it is easy to see that $[k+j]_t=[j]_t+t^j[k]_t$. Thus \eqref{eqn:qjt} reduces to
  $$
    q[j]_t[k+1]_t-qt[k]_t[j-1]_t=q[j]_t+qt^j[k]_t,
  $$
  i.e., $q[j]_t\bigl([k+1]_t-1\bigr)=q[k]_t\bigl(t^j+t[j-1]_t\bigr)$, which is apparent since both quantities simplify to $qt[j]_t[k]_t$.

  Now \eqref{eqn:zkqPartition} follows from \eqref{eqn:znqPartition} by using \eqref{eqn:qjt} to rewrite each factor of the product.
\end{proof}

\begin{prop}[$k$-dependence]\label{lem:kDep}
  Suppose that the integers $q\geq 3$ and $k\geq 1$ and the real number $t\geq 0$ satisfy the tuning equation $qt[k]_t=[2]_t[k+1]_t$ \eqref{eq:qktAna}. Then for all $x\in \bbb{1,q}^m$ and $y\in \bbb{1,q}^n$ we have that
  \begin{equation}
    \sum_{a\in \bbb{1,q}^k}P^{\col}(xay)=P^{\col}(x)P^{\col}(y).
  \end{equation}
\end{prop}
This proposition implies that a coloring with law $\MalCol$ is $k$-dependent whenever $(q,k,t)$ satisfies the tuning equation \eqref{eq:qktAna}. In the proof, we use a well-known $t$-binomial coefficient identity appearing in \cite[{eq.\ (17b)}]{stanley1997enumerative} which states that, for integers $r$ and $s$,
\begin{equation}\label{eq:tBinom}
    \binom{r}{s}_t = \binom{r-1}{s}_t + t^{r-s}\binom{r-1}{s-1}_t.
\end{equation}
\begin{proof}[Proof of Proposition~\ref{lem:kDep}]
  We show that for $x\in \bbb{1,q}^{m}$ and $y\in \bbb{1,q}^n$,
  \begin{equation}\label{eqn:buildingKdep}
    \Bn(x \star^k y)
    =[k]^!_t\left(\frac{q}{[k+1]_t}\right)^k\binom{m+n+2k}{m+k}_t\Bn(x)\Bn(y).
  \end{equation}
  By normalizing both sides of \eqref{eqn:buildingKdep}, it will follow that $P^{\col}(x\star^ky)=c_{m,k,n}P^{\col}(x)P^{\col}(y)$ for some constant $c_{m,k,n}$. But both sides are probability mass functions, so $c_{m,k,n}=1$ and thus the lemma follows directly from \eqref{eqn:buildingKdep}.

  We prove \eqref{eqn:buildingKdep} by induction on $m$ and $n$. The case $m=n=0$ follows from the special case $n=k$ of \eqref{eqn:zkqPartition} in Lemma \ref{lem:zkqt}.

  Suppose that \eqref{eqn:buildingKdep} holds for all words $x$ and $y$ with lengths $m-1$ and $n$ respectively, and for all words $x$ and $y$ with lengths $m$ and $n-1$ respectively. Furthermore, suppose that $x$ and $y$ are proper, since the desired result holds trivially otherwise. Let $a=a_1a_2\cdots a_k$ denote a word; $a$ will be summed over $\bbb{1,q}^k$ below. Applying equation \eqref{eq:Brecurrence2} of Lemma \ref{lem:recurrenceForProb} yields that
  \begin{align*}
    \Bn(xay)&=
    \sum_{i=1}^{m}t^{m+k+n-i}\Bn(\hat x_iay)+\sum_{i=1}^kt^{k+n-i}\Bn(x\hat a_iy)+\sum_{i=1}^nt^{n-i}\Bn(xa\hat y_i)\\
    &\quad-[2]_t\mathbbm{1}[x_m=a_1]t^{k+n-1}\Bn(x\hat a_1y)\\
    &\quad-[2]_t\sum_{i=1}^{k-1}\mathbbm{1}[a_i=a_{i+1}]t^{k+n-i-1}\Bn(x\hat a_iy)\\
    &\quad-[2]_t\mathbbm{1}[a_k=y_1]t^{n-1}\Bn(x\hat a_ky).
  \end{align*}
  Summing over all $a\in \bbb{1,q}^k$ implies that
  \begin{align}
    \Bn(x\star^ky)&=\sum_{i=1}^{m}t^{m+k+n-i}\Bn(\hat x_i\star^ky)+q\sum_{i=1}^kt^{k+n-i}\Bn(x\star^{k-1}y)+\sum_{i=1}^nt^{n-i}\Bn(x\star^k\hat y_i)\nonumber \\
    &\quad-[2]_tt^{k+n-1}\Bn(x\star^{k-1}y)\nonumber \\
    &\quad-[2]_t\sum_{i=1}^{k-1}t^{k+n-i-1}\Bn(x\star^{k-1}y)\nonumber \\
    &\quad-[2]_tt^{n-1}\Bn(x\star^{k-1}y)\nonumber \\
    &=t^{k+n}\sum_{i=1}^m t^{m-i}\Bn(\hat x_i\star^ky)+\sum_{i=1}^n t^{n-i}\Bn(x\star^k\hat y_i)\nonumber \\
    &\quad +\left(qt^m[k]_t-t^{m-1}[2]_t[k+1]_t\right)\Bn(x\star^{k-1}y).\label{eqn:buildingPartialKdep}
  \end{align}
  Note that we have not yet used the assumption that $(q,k,t)$ satisfies the tuning equation \eqref{eq:qktAna}. Crucially, the coefficient of $B_t(x\ \star^{k-1}\ y)$ in \eqref{eqn:buildingPartialKdep} vanishes when the tuning equation is satisfied, so that
  \begin{equation}\label{eq:starSum}
    \Bn(x\star^ky)=t^{k+n}\sum_{i=1}^m t^{m-i}\Bn(\hat x_i\star^ky)+\sum_{i=1}^n t^{n-i}\Bn(x\star^k\hat y_i).
  \end{equation}
  By the inductive hypothesis and equation \eqref{eq:Brecurrence} from Lemma~\ref{lem:recurrenceForProb}, \eqref{eq:starSum} expands to
    \begin{multline*}
      [k]^!_t\left(\frac{q}{[k+1]_t}\right)^k\left[t^{k+n}\binom{m+n+2k-1}{m+k-1}_t\hspace{-1.5mm}+\binom{m+n+2k-1}{m+k}_t\right]\Bn(x)\Bn(y).
    \end{multline*}
    Using \eqref{eq:tBinom}, the above expression simplifies to
    \begin{equation*}
      \Bn(x\star^ky)=[k]^!_t\left(\frac{q}{[k+1]_t}\right)^k\binom{m+n+2k}{m+k}_t\Bn(x)\Bn(y).\qedhere
    \end{equation*}
  \end{proof}

  The final result we will need for Proposition~\ref{prop:sec3main} is a converse to the previous lemma.
  \begin{lemma}\label{lem:converseTuning}
    Let $q\geq 3$ and $t\geq 0$ be given. Suppose that $k$ is a number such that, for all $m\geq 0$ and $n\geq 0$ and for all words $x\in\bbb{1,q}^m$ and $y\in\bbb{1,q}^n$,
    $$
      \sum_{a\in\bbb{1,q}^k}P^{\col}(xay)=P^{\col}(x)P^{\col}(y).
    $$
    Then there exists an integer $k'\leq k$ such that $(q,k',t)$ satisfies the tuning equation \eqref{eq:qktAna}.
  \end{lemma}
  \begin{proof}
    From \eqref{eqn:buildingPartialKdep} in the proof of Proposition \ref{lem:kDep}, if $x$ is a word of length $m$ and $y$ is a word of length $n$ then
    \begin{align*}
      \Bn(x\star^ky)&=t^{k+n}\sum_{i=1}^m t^{m-i}\Bn(\hat x_i\star^ky)+\sum_{i=1}^n t^{n-i}\Bn(x\star^k\hat y_i)\nonumber \\
      &\quad +\left(qt^m[k]_t-t^{m-1}[2]_t[k+1]_t\right)\Bn(x\star^{k-1}y),
    \end{align*}
    which holds for all $q$, $k$, and $t$. In particular, taking $x$ and $y$ to be words of length one, and subtracting two instances of the last equation yields
    $$
      B_t(1\star^k 2)-B_t(1\star^k 1)
      =\left(qt^k[k]_t-t^{k-1}[2]_t[k+1]_t\right)\bigl(B_t(1\star^{k-1} 2)-B_t(1\star^{k-1} 1)\bigr).
    $$
    Upon iterating this identity, we obtain that
    \begin{equation*}\label{eqn:converseQkt}
      \Bn(1\, \star^k\, 2)-\Bn(1\, \star^k\,1)
      =\bigl(\Bn(12)-\Bn(11)\bigr)\prod_{k'=1}^{k}\left(qt^{k'}[k']_t-t^{k'-1}[2]_t[k'+1]_t\right).
    \end{equation*}
    Under our hypotheses on $k$, the left side of the previous equation vanishes. Since $B_t(12)=t+1$ is non-zero but $B_t(11)=0$, one of the factors in the product on the right vanishes.
  \end{proof}

  We conclude this section by proving the proposition stated at the very beginning.
  \begin{samepage}
  \begin{proof}[Proof of Proposition~\ref{prop:sec3main}]
    We show that, for all $t\in [0,1]$ and all $q\geq 3$, the measure $\MalCol_{q,t}$ is the law of a random coloring $X=(X_i)_{i\in\m Z}$ satisfying all conditions in the proposition. Namely,
    \begin{itemize}
      \item for each finite interval $I$, the restricted coloring $(X_i)_{i\in I}$ is equal in law to a uniform $q$-coloring of the constraint graph of a $\BMal_{t,u}$-distributed permutation of $I$, and
      \item the coloring $X$ is strictly $k$-dependent if and only if $(q,k,t)$ satisfies the tuning equation $qt[k]_t=[2]_t[k+1]_t$ \eqref{eq:qktAna}.
    \end{itemize}
    The claims in the first bullet follow from Lemma~\ref{lem:kolm} and the comments involving $\Joint$ preceding that lemma. The second bullet follows by combining Lemma~\ref{lem:kolm} with Proposition~\ref{lem:kDep} and Lemma~\ref{lem:converseTuning}.
  \end{proof}
  \end{samepage}

%%%%%%%%%%%%%%%%%%%%%%%%%%%
%%%%%%%%%%%%%%%%%%%%%%%%%%%
%%                       %%
%%       Section 4       %%
%%       Reversibility   %%
%%                       %%
%%%%%%%%%%%%%%%%%%%%%%%%%%%
%%%%%%%%%%%%%%%%%%%%%%%%%%%

\section{Reversibility}\label{sec:reverse}
  The primary purpose of this section is to prove the following.
  \begin{prop}\label{prop:reverseMain}
    For $q\geq 3$ and $t\in[0,1]$, if $(X_i)_{i\in\m Z}$ has law
    $\MalCol_{q,t}$, then so does $(X_{-i})_{i\in\m Z}$.
  \end{prop}
  This proposition plays an auxiliary role in the proof of Theorem~\ref{main},
  and its proof is technical. Readers eager for the proof of the main theorem
  may safely proceed to the next section.

  The secondary purpose of this section is to justify the following claim from
  the introduction. In fact, it is a corollary of the previous result.
  \begin{cor}\label{cor:1depqcol}
    For each $q\geq 4$, there exists a unique $t\in [0,1]$ such that $(q,1,t)$
    satisfies \eqref{eq:qktAna}. For this $t$ the symmetric $1$-dependent
    $q$-colorings from \cite{HL2} have law $\MalCol_{q,t}$.
  \end{cor}

  Throughout this section we fix a finite interval $I$ of $\m Z$ and a word $x$
  indexed by $I$. Recall from Section~\ref{sec:background} that a permutation
  $\sigma$ of $I$ is a proper building of $x$ if and only if each of the
  subwords
  $$
    x^{\sigma}(t)=\bigl(x_i\colon \sigma(i)\leq t\bigr),\qquad t\in I
  $$
  is proper. The word $x^{\sigma}(t)$ is obtained from $x^{\sigma}(t-1)$ by
  inserting $x_{\sigma^{-1}(t)}$ at position $\sLtilde(\sigma)_t$ from the right.

  For $i,j\in I$ we write $(i\ j)$ for the permutation transposing $i$ and $j$. We also define $\Delta_k(\sigma):=\sLtilde(\sigma)_{k+1}-\sLtilde(\sigma)_k$.
  \begin{lemma}\label{lem:revBij}
    Let $I$ be a finite interval of $\m Z$ and let $k$ be an integer such
    that $k,k+1\in I$.
    Let $\delta$ be an integer and let $
    \ell=(\ell_i\colon i\in I\setminus\{k,k+1\})$ be a sequence.
    Let $A_{\delta,\ell}$ be the set of permutations $\sigma$ of $I$ such that
    restriction of $\sLtilde(\sigma)$ to $I\setminus\{k,k+1\}$ is $\ell$ and
    $\Delta_k(\sigma)=\delta$.
    Then for all words $x\in\bbb{1,q}^I$,
    \begin{equation}\label{eq:revBij}
      \#\bigl\{\sigma\in A_{\delta,\ell}\colon \sigma\vdash x\bigr\}=
      \#\bigl\{\sigma\in A_{\delta,\ell}\colon (k\ k+1)\circ\sigma\vdash x\bigr\}.
    \end{equation}
  \end{lemma}

  We remark that if it were the case that for all graphs $G$
  \begin{equation}\label{eq:revBijGraph}
    \#\{\sigma\in A_{\delta,\ell}\colon \Gamma_\sigma=G\}=
    \#\{\sigma\in A_{\delta,\ell}\colon \Gamma_{(k\ k+1)\circ\sigma}=G\},
  \end{equation}
  then \eqref{eq:revBij} would follow immediately, since $\sigma\vdash x$ is
  equivalent to the assertion that $x$ is a proper coloring of the graph
  $\Gamma_\sigma$.
  As we will see in the proof, \eqref{eq:revBijGraph} does hold in many cases
  but not all.
  For example, it does not hold in the case ${I=\bbb{1,4}}$, $k=3$, and
  ${\delta=\ell_1=\ell_2=0}$ with the graph
  $$
    G=(I,E),\qquad \text{where}\qquad E=\{(1,2),(2,3),(1,3),(3,4),(1,4)\}.
  $$

  We will use the following result in the proof of the lemma.
  \begin{lemma}\label{lem:lehmerSwap}
    Let $I$ be a finite interval, let $k$ be an integer such that
    $\min I\leq k<\max I$,
    and let $\sigma$ be a permutation of $I$. Then
    \begin{alignat*}{4}
      \sLtilde\bigl((k\ k+1)\circ\sigma\bigr) &=\bigl(\ldots,\ell_{k-1},\ &&\ell_{k+1}-\mathbbm{1}[\ell_{k+1}>\ell_k],\ &&\ell_{k}+\mathbbm{1}[\ell_{k+1}\leq\ell_k],\ &&\ell_{k+2},\ldots\bigr),\\
      \text{where}\quad \sLtilde(\sigma)      &=\bigl(\ldots,\ell_{k-1},\ &&\ell_k,                                     &&\ell_{k+1},                                  &&\ell_{k+2},\ldots).
    \end{alignat*}
    In particular, $\Delta_{k}\bigl((k\ k+1)\circ\sigma\bigr)=1-\Delta_k(\sigma)$.
    %$$\Delta_k(\sigma)+\Delta_{k}\bigl((k\ k+1)\circ\sigma\bigr)=1.$$
  \end{lemma}
  The proof is straightforward and is omitted.

  \begin{proof}[Proof of Lemma~\ref{lem:revBij}]
    Fix $x$, $k$, $\delta$, and $\ell$.
    For any permutation $\sigma$ of $I$ and any set $S$ of such permutations, we write
    $$
      \sigma':=(k\ k+1)\circ \sigma,\qquad S'=\{\sigma'\colon \sigma\in S\}.
    $$
    Note that $\sigma''=\sigma$ and $\# S=\# S'$.
    The only difference between $\sigma$ and $\sigma'$ is that the integers
    arriving at times $k$ and $k+1$ are interchanged.
    In particular, $x^\sigma(t)=x^{\sigma'}(t)$ for all $t\not=k$.
    Let $E$ be the set of permutations of $I$ that are proper buildings of $x$.
    Using this notation, \eqref{eq:revBij} may be expressed as
    \begin{equation}\label{eq:card}
      \# A_{\delta, \ell}\cap E  = \# A_{\delta, \ell}\cap E'.
    \end{equation}

    Say that a permutation $\sigma$ is \textbf{almost proper} (at $k$, with
    respect to $x$) if $x^\sigma(t)$ is a proper word for all $t\not=k$, and
    let $P$ be the set of almost proper permutations. Observe that
    $E\subseteq P$ and that $P'=P$, from which it follows that
    $E'\subseteq P$ as well.

    Since $\# A_{\delta, \ell}\cap (P\setminus E ) = %
    \# A_{\delta, \ell}\cap P -  \# A_{\delta, \ell}\cap E $ and similarly
    with $E'$ in place of $E$, \eqref{eq:card} is equivalent to
    \begin{equation}\label{eq:card2}
      \# A_{\delta, \ell}\cap (P\setminus E ) =
      \# A_{\delta, \ell}\cap (P\setminus E' ).
    \end{equation}

    It follows from Lemma~\ref{lem:lehmerSwap} that $(A_{\delta,\ell})' = %
    A_{1-\delta,\ell}$. Thus the set ${\bigl(A_{\delta,\ell}\cap (P\setminus E')\bigr)'}$
    is equal to ${A_{1-\delta, \ell}\cap (P \setminus E )}$, and therefore
    $\# A_{\delta,\ell}\cap (P\setminus E')=\# A_{1-\delta, \ell}\cap (P \setminus E )$.
    Substituting this into equation
    \eqref{eq:card2}, we see that it is equivalent to
    \begin{equation}\label{eq:card3}
      \# A_{\delta, \ell}\cap (P\setminus E ) =
      \# A_{1-\delta, \ell}\cap (P\setminus E ).
    \end{equation}

    We establish \eqref{eq:card3} by exhibiting an involution of the set
    $F:=P\setminus E$ that leaves $\ell$ fixed and interchanges $\delta$ with
    $1-\delta$. Observe that $F$ is the set of permutations $\sigma$ that are
    almost proper but $x^\sigma(k)$ is non-proper.
    Thus for $\sigma\in F$ the word $x^\sigma(k)$ is non-proper but becomes
    proper after deleting a single character.
    This means that there exists a unique pair of integers $r=r^\sigma$ and
    $s=s^\sigma$ such that $r<s$ and $x_r=x_s$ and $r,s$ occur as consecutive
    indices in the sequence $\bigl(m\colon \sigma(m)\leq k\bigr)$ that indexes
    $x^\sigma(k)$.
    Observe that either $\sigma(r)=k$ or $\sigma(s)=k$.
    Let $a=\sigma^{-1}(k+1)$ denote the integer arriving at time $k+1$.
    Since $x^\sigma(k+1)$ is proper, $x_a$ must be inserted between $x_r$ and $x_s$.
    (This implies, in particular, that $\Delta_k(\sigma)\in \{0,1\}$.)

    Define the function $f$ on $F$ via
    $$
      f(\sigma)=\sigma\circ (r^\sigma\ s^\sigma),\qquad \sigma\in F.
    $$
    That is, $f$ interchanges the arrival times of $r$ and $s$ in $\sigma$.
    Since $x_{r^\sigma}=x_{s^\sigma}$ for $\sigma\in F$, it follows that
    ${x^\sigma(t)=x^{f(\sigma)}(t)}$ for all $t\in I$. Thus $f(\sigma)\in F$
    as well. Furthermore $r^{f(\sigma)}=r^{\sigma}$ and similarly for $s$,
    from which it follows that $f$ is involutive.
    All that remains is to verify that $f$ preserves $\ell$ and interchanges
    $\delta$ with $1-\delta$, that is,
    \begin{equation}\label{eq:LtildeEq}
      \sLtilde\bigl(f(\sigma)\bigr)_{i} =
      \sLtilde(\sigma)_{i},\qquad
      i\in I\setminus \{k,k+1\},\ \sigma\in F,
    \end{equation}
    and
    \begin{equation}\label{eq:DeltaK}
      \Delta_k\bigl(f(\sigma)\bigr)=1-\Delta_k(\sigma),\qquad \sigma\in F.
    \end{equation}
    Suppose, for the sake of concreteness, that $\sigma(j)=k$; the other case
    $\sigma(i)=k$ is similar.
    The subwords of $x$ built by $\sigma$ and $f(\sigma)$ at times $k-1$, $k$, and $k+1$ are then
    \begin{alignat*}{4}
       x^\sigma(k-1) &=u\ x_i\ v, \quad & x^\sigma(k)&=u\ x_i\ \text{{\Large \textcircled{\normalsize$x_j$}}}\ \boxed{v},\qquad & x^\sigma(k)&=u\ x_i\ \text{{\Large \textcircled{\normalsize$x_a$}}}\ \boxed{x_j\ v},\qquad &\\
       x^{f(\sigma)}(k-1) &=u\ x_j\ v, \quad & x^{f(\sigma)}(k)&=u\ \text{{\Large \textcircled{\normalsize$x_i$}}}\ \boxed{x_j\ v},\qquad & x^{f(\sigma)}(k)&=u\ x_i\ \text{{\Large \textcircled{\normalsize$x_a$}}}\ \boxed{x_j\ v},\qquad &
    \end{alignat*}
    where we have circled the character that was just inserted, boxed the
    subword to its right, and set $u=\bigl(x_m\colon \sigma(m)<k,\ k<i\bigr)\hspace{0.5em}\text{and}
    \hspace{0.5em}v=\bigl(x_m\colon \sigma(m)<k,\ m>j\bigr).$
    Thus when $\sigma(j)=k$, we have that
    $$
      \Delta_k(\sigma)=|x_jv|-|v|=1\hspace{0.5em}\text{and}
      \hspace{0.5em}\Delta_k\bigl(f(\sigma)\bigr)=|x_jv|-|x_jv|=0.
    $$
    Similarly when $\sigma(i)=k$, we have that
    $$
      \Delta_k(\sigma)=|x_jv|-|x_jv|=0\hspace{0.5em}\text{and}
      \hspace{0.5em}\Delta_k\bigl(f(\sigma)\bigr)=|x_jv|-|v|=1.
    $$
    This establishes \eqref{eq:DeltaK}.
    Moreover \eqref{eq:LtildeEq} clearly holds when $i>k+1$, and for $i<k$ it
    follows since $x_r$ and $x_s$ occupy the same relative positions in the
    respective subwords $x^\sigma(i)$ and $x^{f(\sigma)}(i)$.
  \end{proof}

  Fix $n\geq 0$. For a word $x=(x_i)_{i=0}^n$, we write $\overline x=(x_{n-i})_{i=0}^n$ for its reversal. For a permutation $\sigma\in \Sym(\bbb{0,n})$, its reversal is the permutation $\overline\sigma$ with $\overline\sigma(i)=\sigma(n-i)$.
  \begin{lemma}\label{cor:malProb}
    Let $\sigma$ be a Mallows-distributed permutation of $\bbb{0,n}$. Then
    $$\m P(\sigma\vdash x)=\m P(\overline\sigma\vdash x),\qquad \forall x\in\bbb{1,q}^{\bbb{0,n}}.$$
  \end{lemma}
  \begin{proof}
    Write $\sG_i$ and $\overline{\sG}_i$ for the respective laws of $X_i$ and $i-X_i$, where $0\leq i\leq n$ and $X_i$ is an $i$-truncated, $t$-geometric random variable. By Lemma~\ref{lem:LehmerBub}, the random sequences $\sLtilde(\sigma)$ and $\sLtilde(\overline{\sigma})$ have laws $\sG_0\otimes \sG_1\otimes \cdots\otimes \sG_n$ and $\overline{\sG_0}\otimes \cdots \otimes \overline{\sG_n}$, respectively. Fix a word $x$ and let $B_x$ denote the set of tuples $\ell\in\{0\}\times \cdots \times \bbb{0,n}$ such that $\sDtilde(\ell)$ is a proper building of $x$. The desired result is equivalent to
    \begin{equation}\label{eq:prodGeomLaw}
      \sG_0\otimes \cdots\otimes \sG_n(B_x)=\overline{\sG_0}\otimes \cdots\otimes \overline{\sG_n}(B_x).
    \end{equation}
    Note that $\sG_0=\overline{\sG_0}$. We will establish \eqref{eq:prodGeomLaw} by showing that
    \begin{align*}
      \sG_0\otimes \cdots\otimes \sG_n(B_x) &=\overline{\sG_0}\otimes           \sG_1 \otimes \cdots\otimes           \sG_{n-1}  \otimes           \sG_n (B_x)\\
                                            &=          \sG_0 \otimes \overline{\sG_1}\otimes \cdots\otimes           \sG_{n-1}  \otimes           \sG_n (B_x)\\[2pt]
      \cdots\qquad                          &=          \sG_0 \otimes           \sG_1 \otimes \cdots\otimes           \sG_{n-1}  \otimes \overline{\sG_n}(B_x)\\[1em]
      \cdots\qquad                          &=          \sG_0 \otimes           \sG_1 \otimes \cdots\otimes \overline{\sG_{n-1}} \otimes \overline{\sG_n}(B_x)\\[2pt]
                                            &                                      \hspace{1.5cm} \cdots\cdots\cdots                                          \\[2pt]
      \cdots\qquad                          &=\overline{\sG_0}\otimes \overline{\sG_1}\otimes \cdots\otimes \overline{\sG_{n-1}} \otimes \overline{\sG_n}(B_x).
    \end{align*}
    More precisely, we will show that for all $0\leq k<m\leq n$ fixed,
    \begin{equation}\label{eq:kmnGeom}
      \mu(B_x)=\mu'(B_x),
    \end{equation}
    where
    \begin{align*}
       \mu &=\sG_0\otimes \cdots\otimes \sG_{k-1}\otimes \overline{\sG_k}\otimes           \sG_{k+1} \otimes\sG_{k+2}\otimes \cdots \otimes \sG_{m-1}\otimes \overline{\sG_m}\otimes \cdots\otimes \overline{\sG_n}\hspace{0.5em}\text{and}\\
       \mu'&=\sG_0\otimes \cdots\otimes \sG_{k-1}\otimes           \sG_k \otimes \overline{\sG_{k+1}}\otimes\sG_{k+2}\otimes \cdots \otimes \sG_{m-1}\otimes \overline{\sG_m}\otimes \cdots\otimes \overline{\sG_n}.%
    \end{align*}

    Any two independent truncated geometric random variables are conditionally uniform on the set of possible values given their difference.
    From this it follows that a random tuple $L$ with law $\mu$ is conditionally uniform on some set given $L_{k+1}-L_k$ and ${\bigl(L_i\colon i\in \bbb{0,n}\setminus \{k,k+1\}\bigr)}$.
    Thus by Lemma~\ref{lem:revBij}, the following conditional probabilities are equal:
    $$
    \mu\bigl(B_x\mid L_{k+1}-L_k, (L_i\colon i\not=k,k+1)\bigr)=\mu\bigl(B_x'\mid L_{k+1}-L_k, (L_i\colon i\not=k,k+1)\bigr),
    $$
    where
    $$
      B_x':=\Bigl\{\ell\in \{0\}\times \cdots\times \bbb{0,n}\colon (k\ k+1)\circ\sDtilde(\ell)\vdash x\Bigr\}.
    $$
    Integrating out the conditioning yields that
    \begin{equation}\label{eq:eqBxPrime}
      \mu(B_x)=\mu(B_x').
    \end{equation}
    If $(L_k,L_{k+1})$ has law $\overline{\sG_k}\otimes \sG_{k+1}$, then
    $$
      \bigl(L_{k+1}-\mathbbm{1}[L_{k+1}>L_{k}],\ L_k+\mathbbm{1}[L_{k+1}\leq L_{k}]\bigr)
    $$
    has law $\sG_k\otimes \overline{\sG_{k+1}}$.
    When combined with Lemma~\ref{lem:lehmerSwap}, this yields that $\mu(B_x')=\mu'(B_x)$.
    We now deduce \eqref{eq:kmnGeom} from this and \eqref{eq:eqBxPrime}.
    The lemma follows by repeated application of \eqref{eq:kmnGeom} as indicated above.
  \end{proof}

  \begin{cor}\label{cor:obv}
    If $X=(X_i)_{i\in\m Z}$ is a random $q$-coloring of $\m Z$ with law $\MalCol_{q,t}$, then
    $$
      (X_0,X_1,\ldots,X_n)\eqd(X_n,X_{n-1},\ldots,X_0),\qquad n\geq 0.
    $$
  \end{cor}
  \begin{proof}
    This follows directly from the previous lemma, since for any word $x$ and for $\sigma$ a random Mallows permutation, equation \eqref{eq:malPcol} implies that the quantities $\m P(\sigma\vdash x)$ and $\m P\bigl((X_0,\ldots,X_n)=x\bigr)$ are constant multiples of one another.
  \end{proof}

  \begin{proof}[Proof of Proposition~\ref{prop:reverseMain}]
    This is immediate from Corollary~\ref{cor:obv}.
  \end{proof}

\begin{proof}[Proof of Corollary~\ref{cor:1depqcol}]
  By inspection, $(q,t,1)$ satisfies \eqref{eq:qktAna} if and only if $q=(t+1)^2/t$, which is equivalent to
  \begin{equation}\label{eq:q1t}
    t^{1/2}+t^{-1/2}=\sqrt{q}.
  \end{equation}
  Provided $\sqrt{q}\geq 2$, there exists $t=t(q)\in[0,1]$ satisfying this equation. 
  Then for any proper word $x$,
  $$
    \Bn(x)=\sum_{i=0}^n t^{n-i}\Bn(\hat x_i) \quad \text{ and } \quad \Bn(x)=\sum_{i=0}^n t^{i}\Bn(\hat x_i),
  $$
  where the first equation is due to Lemma~\ref{lem:recurrenceForProb} and the second follows by combining Lemma~\ref{lem:recurrenceForProb} with Corollary~\ref{cor:obv}.
  Averaging these two equalities yields
  \begin{equation*}\label{eqn:symmetrReq}
    \Bn(x)=\sum_{i=0}^n \left(\frac{t^{n-i}+t^i}{2}\right)\Bn(\hat x_i).
  \end{equation*}
  By \eqref{eq:q1t}, the coefficients of this recurrence agree with those of \cite[Eq. 2.2]{HL2}.
\end{proof}

%%%%%%%%%%%%%%%%%%%%%%%%%%%
%%%%%%%%%%%%%%%%%%%%%%%%%%%
%%                       %%
%%       Section 5       %%
%%       Lehmer conv.    %%
%%                       %%
%%%%%%%%%%%%%%%%%%%%%%%%%%%
%%%%%%%%%%%%%%%%%%%%%%%%%%%

\section{Convergence of Lehmer codes}\label{sec:bubbles}
  Previously we showed that $\MalCol_{q,t}$ is $k$-dependent whenever $(q,k,t)$ satisfy the tuning equation \eqref{eq:qktAna} and that $\MalCol_{q,t}$ is reversible. The key step remaining in the proof of the main theorem is to exhibit a finitary factor of an iid process having this law. We accomplish this over the course of the next two sections. The idea will be to give a probabilistic construction of the colorings on $\m Z$. On a finite interval, this was done already in Proposition~\ref{prop:sec3main}. Our goal now is to extend this construction to $\m Z$ by taking appropriate limits.

  In this section we show that the law of the Lehmer code of a  $\BMal$-distributed permutation of $\bbb{-n,n}$ is iid in the limit as $n\to\infty$ with a certain distribution. In the next section we will use this iid sequence to produce a random coloring of $\m Z$ satisfying all of the properties claimed in Theorem~\ref{main}.

  Both the Lehmer code, $\sL$, and the insertion code, $\sLtilde$, play an important role in these sections. The zeros of the Lehmer code occur at key locations in the coloring (essentially, they are the renewal times). On the other hand, the pushforward of $\BMal$ under the insertion code is the product of truncated geometric distributions. The pushforward under the Lehmer code is not a product measure for general finite intervals, but our main result in this section is that it tends towards a product measure as the interval approaches $\m Z$.

  Recall that the Lehmer code is the map $\sL\colon \Sym(I)\to \bbb{0,\infty}^{I}$ given by
  $$
  \sL(\sigma)=\Bigl(\#\bigl\{ j\in I\colon j>i\text{ and }\sigma(j)<\sigma(i)\bigr\}\Bigr)_{i\in I}\quad \text{for } \sigma\in\Sym(I).
  $$
  Here we identify $\bbb{0,\infty}^I$ with the subset of $\bbb{0,\infty}^{\m Z}$ consisting of sequences vanishing on $\m Z\setminus I$, so that we can compare permutations on different intervals.

  Recall from Section~\ref{sec:background} that $\BMal_{t,u}$, the bubble-biased Mallows measure on permutations, assigns to $\sigma\in\Sym(I)$ a probability proportional to $u^{\#\!\bub(\Gamma_\sigma)}t^{\inv(\sigma)}$.

  \begin{prop}\label{prop:codingConvergence}
    Let $\sigma_n$ be a random permutation of $\bbb{-n,n}$ with law $\BMal_{t,u}$. As $n\to\infty$ the sequence $\sL(\sigma_n)$ converges in law, with respect to the product topology on $\bbb{0,\infty}^{\m Z}$, to an iid sequence of $u$-zero-weighted, $t$-geometric random variables.
  \end{prop}

   The remainder of this section is devoted to the proof of this proposition, which requires several simple but technical lemmas. Recall that $\c F(\sigma)$ denotes the set of founders of the permutation $\sigma$ (see Section~\ref{sec:background}).

  \begin{lemma}\label{lem:bubEndpointCoding}
    Let $\sigma$ be a permutation of $\bbb{0,n}$. Then $\sigma^{-1}(0)$ is a founder of $\sigma$ and
    \begin{equation}\label{eqn:firstZero}
      \sigma^{-1}(0)=\min \{0\leq i\leq n\colon \mathscr{L}(\sigma)_i=0\}\text{ and}
    \end{equation}
    \begin{equation}\label{eqn:bubbleSets}
    \c F(\sigma)\cap \bbb{\sigma^{-1}(0),n}=\{i\in I\colon \mathscr{L}(\sigma)_i=0\}\text{ and}
    \end{equation}
    \begin{equation}\label{eqn:bubbleSets2}
    \c F(\sigma)\cap \bbb{0,\sigma^{-1}(0)}=\{i\in I\colon \mathscr{L}(\sigma)_i=\sigma(i)\}.
    \end{equation}
  \end{lemma}

  \begin{proof}
    By the definition of a founder, $i\in \c F(\sigma)$ if and only if either $\sigma(k)>\sigma(i)$ for all $k>i$ or $\sigma(j)>\sigma(i)$ for all $j<i$. When $i=\sigma^{-1}(0)$ both conditions hold. If $i> \sigma^{-1}(0)$, then the latter condition cannot hold, and so
    $$
      \c F(\sigma)\cap \bbb{\sigma^{-1}(0),n}=\{i\in I\colon  \sigma(k)>\sigma(i) \text{ for all }k>i\}=\{i\in I\colon \mathscr{L}(\sigma)_i=0\}.
    $$
    Similarly if $i<\sigma^{-1}(0)$, then the former condition cannot hold, and so
    \begin{align*}
      \c F(\sigma)\cap \bbb{0,\sigma^{-1}(0)}&=\{i\in I\colon  \sigma(j)>\sigma(i) \text{ for all } j<i\}=\{i\colon \mathscr{L}(\sigma)_i=\sigma(i)\},
    \end{align*}
    where the last equality follows since $\sigma(i)-\mathscr{L}(\sigma)_i=\#\{j<i\colon \sigma(j)<\sigma(i)\}$.
  \end{proof}

  \begin{lemma}\label{lem:conditionalDist}
    Let $\sigma$ be a random permutation of $\bbb{m,n}$ with law $\BMal_{t,u}$ and let $i\in \bbl{m,n}$. Given $\sigma^{-1}(m)<i$, the random variables $(\sL(\sigma)_j)_{j=i}^n$ are conditionally independent of each other, with the conditional law of $\sL(\sigma)_j$ being the $u$-zero-weighted, $(n-j)$-truncated, $t$-geometric distribution.
  \end{lemma}
  \begin{proof}
    By a simple relabelling we assume that $m=0$ without loss of generality.
    For each $\ell = %
    (\ell_i,\ldots,\ell_n) \in \bbb{0,n-i}\times\cdots \times \{0\}$, let
    \begin{align*}
      \mathscr{A}(\ell) =
        \bigl\{ \tau \in \Sym(\bbb{0,n}):  \, \tau^{-1}(0) < i \text{ and }
        \mathscr{L}(\tau)_j=\ell_j \text{ for all } i\leq j\leq n\bigr\}.
    \end{align*}

    Recall that $\#\c F(\tau)=\#\!\bub(\Gamma_\tau)+1$. Thus
    \begin{align*}
      \m P\bigl((\mathscr{L}(\sigma)_i,\ldots,\mathscr{L}(\sigma)_n)=\ell
       \mid \sigma^{-1}(0)<i\bigr) =
      \frac{\m P\bigl(\sigma\in \sA(\ell)\bigr)}
        {\m P \bigl( \sigma^{-1}(0) < i\bigr )}
      &=\frac{1}{Z}\sum_{\tau \in \mathscr{A}(\ell)
      }t^{\inv(\tau)}u^{\#\c F(\tau)},
    \end{align*}
    where $Z=\m P \bigl( \sigma^{-1}(0) < i\bigr )$.

    Observe that
    \begin{align}
      \sum_{\tau \in \mathscr{A}(\ell)}t^{\inv(\tau)}u^{\#\c F(\tau)} &=
      \sum_{\tau \in \mathscr{A}(\ell)}\prod_{j=0}^nt^{\mathscr{L}(\tau)_j}u^{\mathbbm{1}[j\in \c F(\tau)]}\nonumber\\
      &=\prod_{j=i}^{n}t^{\ell_j}u^{\mathbbm{1}[\ell_j=0]}
      \Biggl[\sum_{\tau \in \mathscr{A}(\ell)}\prod_{j=0}^{i-1}t^{\mathscr{L}(\tau)_j}u^{\mathbbm{1}[j\in \c F(\tau)]}\Biggr],\label{eq:bubbleFactorCondition}
    \end{align}
    where in the second equality we have applied equation \eqref{eqn:bubbleSets} of Lemma \ref{lem:bubEndpointCoding}. Let
    \begin{equation*}
      R_{u,n,i,t}(\ell):=\sum_{\tau \in \mathscr{A}(\ell)}\prod_{j=0}^{i-1}t^{\mathscr{L}(\tau)_j}u^{\mathbbm{1}[j\in \c F(\tau)]}
    \end{equation*}
    denote the bracketed expression in \eqref{eq:bubbleFactorCondition}. The product appearing in $R_{u,n,i,t}$ factorizes into two factors, corresponding to equations \eqref{eqn:bubbleSets} and \eqref{eqn:bubbleSets2} of Lemma \ref{lem:bubEndpointCoding} respectively:
    \begin{equation}\label{eqn:RfactorSum}
        R_{u,n,i,t}(\ell)= \sum_{\tau \in \mathscr{A}(\ell)} \Biggl[\prod_{j=0}^{\tau^{-1}(0)-1} t^{\mathscr{L}(\tau)_j}u^{\mathbbm{1}[\mathscr{L}(\tau)_j=\tau(j)]} \Biggr] \Biggl[\prod_{j=\tau^{-1}(0)}^{i-1} t^{\mathscr{L}(\tau)_j}u^{\mathbbm{1}[\mathscr{L}(\tau)_j=0]} \Biggr].
    \end{equation}
  We wish to show that the function $R_{u,n,i,t}$ is constant, i.e., that for
  any pair of tuples $\ell$ and $\ell'$ belonging to
  $\bbb{0,n-i}\times\cdots\times \{0\},$
  it holds that $R_{u,n,i,t}(\ell)=R_{u,n,i,t}(\ell')$.

  Recall that $\sL$ is a bijection from $\Sym(\bbb{0,n})$ to
  $\bbb{0,n}\times\cdots\times\{0\}$ with inverse $\sD$.
  For any $\tau\in\sA(\ell)$, eq.~\eqref{eqn:bubbleSets} implies that
  $\tau^{-1}(0)<i$ if and only if there exists $j\in \bbr{0,i}$ such that
  $\mathscr{L}(\tau)_j=0$. Thus $\sL(\sA(\ell))$ is the set of tuples in
  $\bbb{0,n}\times\cdots \times \{0\}$ whose restriction to $\bbb{i,n}$ is
  $\ell$ and whose restriction to $\bbr{0,i}$ has at least one entry that
  vanishes. The set $\sL(\sA( \ell'))$ bears a similar description.
  Let $P$ be the bijection from $\sL(\sA(\ell))$ to $\sL(\sA(\ell'))$ leaving
  the restriction to $\bbr{0,i}$ fixed, and let $Q=\sD\circ P\circ \sL$ be the
  corresponding bijection from $\sA(\ell)$ to $\sA(\ell')$.

  By the explicit formula for $\sD$ in \eqref{eq:D}, it follows that
  $$
    \tau(j)=\left(\bigCircle_{k=0}^j\pi^-_{\bbb{k,k+\sL(\tau)_k}}\right)
    \hspace{-.3em}(j),\qquad j\in\bbr{0,i}.
  $$
  Since $P$ fixes the restriction of $\sL(\tau)$ to $\bbr{0,i}$, we have that
  $Q$ fixes the restriction of $\tau$ to $\bbr{0,i}$.
  Applying \eqref{eqn:bubbleSets} once more shows that $Q$
  also fixes $\tau^{-1}(0)$. Thus, each summand of \eqref{eqn:RfactorSum}
  is unchanged by the action of $Q$, from which it follows that
  $R_{u,n,i,t}(\ell)=R_{u,n,i,t}(\ell')$. Since $\ell$ and $\ell'$ were
  arbitrary, $R_{u,n,i,t}$ does not depend on $\ell$. The lemma now follows from
  \eqref{eq:bubbleFactorCondition}.
  \end{proof}

  A real-valued random variable $X$ is said to
  \textbf{stochastically dominate} another random variable $Y$ if
  $\m P(X>r)\geq \m P(Y>r)$ for all $r$.
  \begin{lemma}\label{lem:domination}
    Fix $0<t<s<1$ and $u\geq 1$. Let $S$ be an
    $n$-truncated, $s$-geometric random variable and let $T$ be a
    $u$-end-weighted, $n$-truncated, $t$-geometric random variable.
    Then $S$ stochastically dominates $T$ for all
    \begin{equation*}
        n\geq n_0:=\log_{s/t}\left(u\cdot \frac{1-t}{1-s}\right).
    \end{equation*}
  \end{lemma}
  For terminology regarding variants of geometric random variables,
  refer to equations (i) to (v) in Section~\ref{sec:background}.
  \begin{proof}%[Proof of Lemma \ref{lem:domination}]
  Let $M$ be a $u$-max-weighted $n$-truncated $t$-geometric variable.
  By inspection of the mass functions of $M$ and $T$, it is apparent
  that $M$ stochastically dominates $T$ for all $u\geq 1$. Thus it remains to
  show that $S$ stochastically dominates $M$.

  For this, we argue that $\m P(M< k)\geq \m P(S< k)$ for all $k\in\bbl{1,n}$,
  which by the formulas in (i) and (iii) of Section~\ref{sec:background} is
  equivalent to
  $$
    \frac{\frac{1-t^{k}}{1-t}}{\frac{1-t^n}{1-t}+ut^n}\geq
    \frac{1-s^{k}}{1-s^{n+1}},
  $$
  which is, in turn, equivalent to
  $$\frac{1-t^{k}}{1-s^{k}}\geq
    \frac{1-t^n+ut^n(1-t)}{1-s^{n+1}}.
  $$
  Since $t<s$, the left side of the latter inequality is a decreasing
  function of $k$, and thus it suffices to prove the inequality for $k=n$.
  In this case the inequality rearranges to
  \begin{align*}
    \left(\frac{s}{t}\right)^n\geq u\cdot \frac{1-t}{1-s}\cdot \frac{1-s^n}{1-t^n}.
  \end{align*}
  By our choice of $n_0$, the inequality holds for all $n\geq n_0$, as desired.
\end{proof}

We use the previous lemma to prove the following tightness result for the
bubble-biased Mallows measure.
\begin{lemma}\label{lem:tightness}
  Fix $u\in [1,\infty)$ and $t\in [0,1)$. For each $n\geq 0$ let $\sigma_n$ be
  a random permutation $\bbb{0,n}$ with law $\BMal_{t,u}$. Then
  $\bigl(\sigma_n^{-1}(0)\bigr)_{n\geq 0}$ is tight.
\end{lemma}

\begin{proof}
  We prove the result by finding a coupling of the permutations in which
  $$
    \m P\bigl({\textstyle\sup_n \sigma_n^{-1}(0)=\infty}\bigr)=0,
  $$ from which tightness follows.

  Let $(X_n^\sigma)_{n\geq 0}$ be a independent random variables with
  $X_n^\sigma$ being $u$-end-weighted, $n$-truncated, and $t$-geometric.
  By Lemma~\ref{lem:LehmerBub}, for all $n$ the law of the
  random permutation $\sigma_n=\sDtilde(X_0^\sigma,\ldots,X_n^\sigma)$ is
  $\BMal_{t,u}^{\bbb{0,n}}$. By inspection of the formula \eqref{eq:Dtilde}
  for $\sDtilde$, it follows that the sequence
  $\bigl\{\sigma_n^{-1}(0)\bigr\}_{n}$ is a.s.\! non-decreasing, and therefore
  \begin{equation}\label{eq:supInvPerm}
    \m P\bigl({\textstyle\sup_n\sigma_n^{-1}(0)\geq k}\bigr)=
    \lim_{n\to\infty}\m P\bigl(\sigma_n^{-1}(0)\geq k\bigr),
    \qquad \forall k\geq 0.
  \end{equation}
  When $u=1$, the law of $\sigma_n$ is $\Mal_t$, and the same holds for
  $\sigma_n^{-1}$ by inversion symmetry of the Mallows measure. Since
  $\sigma_n^{-1}(0)=\sL(\sigma_n^{-1})_0$, it follows by Lemma~\ref{lem:geomMal}
  that $\sigma_n^{-1}(0)$ is an $n$-truncated, $t$-geometric random variable.
  Combining this with \eqref{eq:supInvPerm} implies that
  $\sup_n\sigma_n^{-1}(0)$ is a.s.\! finite when $u=1$.

  We extend this result from the case $u=1$ to $u>1$ using a domination
  argument. Fix $s$ such that $t<s<1$ and take $n_0$ to be any integer larger
  than the constant $n_0(s,t,u)$ from Lemma \ref{lem:domination}, thereby
  guaranteeing that an $s$-geometric $n$-truncated random variable
  stochastically dominates a $u$-end-weighted, $n$-truncated, $t$-geometric
  random variable for all $n\geq n_0$. Let $(X_n^\tau)_{n\geq 0}$ be independent
  random variables with $X_n^\tau$ being $n$-truncated and $s$-geometric.
  We couple $(X_n^\sigma)_{n>n_0}$ and $(X_n^\tau)_{n>n_0}$ such that
  $X_n^\sigma\leq X_n^\tau$ a.s.\! for all $n>n_0$ using Strassen's
  Theorem \cite{strassen1965existence}. By applying Lemma~\ref{lem:LehmerBub}
  with $(t,u)=(s,1)$, it follows that for all $n\geq 0$, the random permutation
  $\tau_n=\sDtilde(X_0^\tau,\ldots,X_n^\tau)$ has law $\Mal_s$. Hence
  $\sup_n\tau_n^{-1}(0)$ is a.s.\! finite by the previous analysis of the
  case $u=1$.

  For all $i\in\bbr{0,\infty}$ and all $a_i\in\bbb{0,i}$, the probabilities
  $\m P(X_i^\sigma=a_i)$ and $\m P(X_i^\tau=a_i)$ are positive. We claim that,
  for any sequence of integers $a_0,\ldots,a_{n_0}$ satisfying these conditions,
  we have that
  \begin{multline}\label{eq:condBnd}
    \m P\Bigl({\textstyle\sup_n \sigma_n^{-1}(0)=\infty}
      \Bigm| X_i^\sigma=a_i,\ \forall\ 0\leq i\leq n_0\Bigr)\\
      \leq \m P\Bigl({\textstyle\sup_n \tau_n^{-1}(0)=\infty}
            \Bigm| X_i^\tau=a_i,\ \forall\ 0\leq i\leq n_0\Bigr).
  \end{multline}
  Indeed, this is equivalent to
  \begin{multline*}\label{eq:condBnd2}
    \m P\Bigl({\textstyle\sup_n
      \sDtilde(a_0,\ldots,a_{n_0},X_{n_0+1}^\sigma,\ldots,X_n^\sigma)^{-1}(0)
      =\infty}\Bigr)\\
    \leq \m P\Bigl({\textstyle\sup_n
      \sDtilde(a_0,\ldots,a_{n_0},X_{n_0+1}^\tau,\ldots,X_n^\tau)^{-1}(0)
      =\infty}\Bigr),
  \end{multline*}
  which follows since the function $\sDtilde(\ell_0,\ldots,\ell_n)^{-1}(0)$ is
  non-decreasing in each of its arguments, as seen by inspection of the formula
  \eqref{eq:Dtilde} for $\sDtilde$.

  But we have already shown that
  $\m P\bigl({\textstyle\sup_n \tau_n^{-1}(0)=\infty}\bigr)=0,$
  whereupon the right side of \eqref{eq:condBnd} is zero. Thus the left side
  is zero as well. This implies that
  ${\m P\bigl({\textstyle\sup_n \sigma_n^{-1}(0)=\infty}\bigr)}=0,$
  from which the lemma now follows.
  \end{proof}

 \begin{proof}[Proof of Proposition \ref{prop:codingConvergence}]
    Let $L=(L_i)_{i\in\m Z}$ be a sequence of iid $u$-zero-weighted,
    $t$-geometric random variables. For each $n$, let $\sigma_n$ be a random
    permutation of $\bbb{-n,n}$ with law $\BMal_{t,u}$. The statement of the
    proposition is equivalent to the equality
    \begin{equation}\label{eq:pointwisePMF}
      \lim_{n\to\infty}\m P\Bigl(\bigl(\sL(\sigma_n)_i\bigr)_{i\in I}=
      \ell\Bigr)=\m P\bigl((L_i)_{i\in I}=\ell\bigr),
    \end{equation}
    for all finite intervals $I\subset \m Z$ and for all
    $\ell\in\bbb{0,\infty}^I$.

    By considering the interval $\bbb{0,2n}$ and shifting,
    Lemma~\ref{lem:tightness} implies that
    $$
      \lim_{n\to\infty}\m P\bigl(\sigma_n^{-1}(-n)<\min I\bigr)=1
    $$
    Combining this with Lemma~\ref{lem:conditionalDist} establishes
    \eqref{eq:pointwisePMF}, proving the proposition.
  \end{proof}

%%%%%%%%%%%%%%%%%%%%%%%%%%%
%%%%%%%%%%%%%%%%%%%%%%%%%%%
%%                       %%
%%       Section 6       %%
%%       Main theorem    %%
%%                       %%
%%%%%%%%%%%%%%%%%%%%%%%%%%%
%%%%%%%%%%%%%%%%%%%%%%%%%%%

\section{Proof of main theorem} \label{sec:coloringIntegers}
    Fix $t\in (0,1)$ and $q\geq 3$.
    In Section~\ref{sec:finiteMallowsColoring} we constructed a measure
    $\MalCol_{q,t}$ which is the law of a $k$-dependent $q$-coloring of
    $\m Z$ whenever $(q,k,t)$ satisfies the tuning equation \eqref{eq:qktAna}.
    Here we give a construction, directly on the integers, of a random coloring
    with this law.
    This random coloring will arise as a uniform proper $q$-coloring of a
    certain random infinite graph.
    Thus, we begin by explaining what we mean by a
    `uniform proper $q$-coloring' of an infinite graph.

    First suppose that the graph in question is the nearest-neighbor graph on
    the integers. We define a uniform proper $q$-coloring of this graph to be
    the bi-infinite trajectory of a stationary simple random walk on the
    complete graph (without self-loops) with vertex set $\bbb{1,q}$.

    For a graph $G$ with vertex set $\m Z$, recall that the bubble endpoints of
    $G$ were defined in Section~\ref{sec:background} to be those integers $i$
    such that there do not exist integers $j$ and $k$ adjacent in $G$ with
    $j<i<k$. Say that $G$ is \textbf{good} if it is $q$-colorable, its set of
    bubble endpoints is unbounded from above and below, and consecutive bubble
    endpoints are adjacent in $G$. For any good graph $G$, we define a
    \textbf{uniform proper $\bm{q}$-coloring} of $G$ to be a random coloring
    equal in law to the output of the following algorithm.

    \medskip
    \noindent\textbf{Uniform coloring algorithm.}
      Input: a good graph $G$ with vertex set $\m Z$.

      \begin{enumerate}[(i)]
        \itemsep0.7em
        \item \label{alg:col:stage0}
          Let $(b_e)_{e\in\m Z}$ be an increasing enumeration of the
          bubble endpoints of $G$.

        \item \label{alg:col:stage1}
          Let $(Y_e)_{e\in \Z}$ be a uniform proper $q$-coloring of $\Z$, and
          set $X_{b_e}=Y_e$ for all $e\in \Z$.

        \item  \label{alg:col:stage2}
          Conditional on step \eqref{alg:col:stage1}, for each bubble of $G$,
          choose independently a proper $q$-coloring of the bubble uniformly
          from among those that assign the colors from step
          \eqref{alg:col:stage1} to the bubble endpoints. Output the
          resulting coloring $(X_i)_{i\in\m Z}$.
      \end{enumerate}

    \begin{lemma}\label{lem:unifQcol}
      Let $G$ be a good graph. Let $(X_i)_{i\in \m Z}$ be a uniform proper
      $q$-coloring of $G$. Then for every finite interval $I\subset \m Z$ whose
      endpoints are bubble endpoints of $G$, the coloring $(X_i\colon i\in I)$
      is distributed uniformly on the set of proper $q$-colorings of the
      subgraph of $G$ induced by $I$.
    \end{lemma}

    \begin{proof}
      Write $G_I$ for the subgraph of $G$ induced by $I$.
      Let $(b_0,b_1,\ldots,b_n)$ be an increasing enumeration of the bubble
      endpoints of $G$ contained in $I$, and let $\bub_i$ be the subgraph of
      $G$ induced by $\bbb{b_{i-1},b_i}$ for $1\leq i\leq n$.
      Then $G_I=\bub_1\cup \cdots\cup \bub_n$, and there is a bijection between
      proper $q$-colorings of $G_I$ and proper $q$-colorings of
      $\bub_1,\ldots,\bub_n$ that agree at their endpoints.

      Thus if $U=(U_i\colon i\in I)$ is a uniform proper $q$-coloring of $G_I$,
      then the conditional law of $U$ given $U_{b_0},U_{b_1},\ldots,U_{b_n}$
      coincides with that of $(X_i\colon i\in I)$ given
      $X_{b_0},X_{b_1},\ldots,X_{b_n}$, by step \eqref{alg:col:stage2} of the
      algorithm. Since the subgraph of $G$ induced by $b_0,b_1,\ldots,b_n$ is a
      path, the laws of $(X_{b_0},X_{b_1},\ldots,X_{b_n})$ and
      $(U_{b_0},U_{b_1},\ldots,U_{b_n})$ are equal by step
      \eqref{alg:col:stage1} of the algorithm. Thus the unconditional laws of
      $U$ and $(X_i\colon i\in I)$ coincide.
    \end{proof}

    Recall the definition of the constraint graph $\Gamma_\sigma$ of a
    permutation $\sigma$ of $\m Z$. In the case when $\sigma$ is a finite
    permutation of $\m Z$, the graph $\Gamma_\sigma$ can be expressed in terms of the
    Lehmer code $\sL(\sigma)$, a sequence in which all but finitely many
    entries vanish. We now extrapolate to a graph defined in terms of a more
    general sequence. Recall the map $\sD$, inverse to $\sL$, defined in
    Section~\ref{sec:background}.

    \begin{dfn*}[of $\Gamma\lbrack\ell\rbrack$]
      Let $\ell\in\bbr{0,\infty}^{\m Z}$ be a sequence for which the zero set
      $(i\in\m Z\colon \ell_i=0)$ is unbounded from above and below.
      Let $(i_k)_{k\in\m Z}$ be an increasing enumeration of the zero set of $\ell$,
      and for each $k\in \m Z$ let $A_k=(\ell_i\colon i\in \bbb{i_k,i_{k+1}})$.
      Then the graph $\Gamma[\ell]$ is defined to be
      $$
        \Gamma[\ell]:=\bigcup_{k\in\m Z}\Gamma_{\sD(A_k)}.
      $$
    \end{dfn*}

    This generalizes the definition of the constraint graph in that,
    if $\sigma$ is a finite permutation of $\m Z$, then
    $\Gamma\bigl[\sL(\sigma)\bigr]=\Gamma_\sigma$. In fact, this follows from
    the next lemma.

    \begin{lemma}\label{lem:localBub}
      For any integers $a\leq i<j\leq b$ and for any sequence
      $\ell\in\bbr{0,\infty}^{\m Z}$ with zero set unbounded from above and
      below, the integers $i$ and $j$ are adjacent in $\Gamma[\ell]$ if and
      only if they are adjacent in the constraint graph of
      $\sD(\ell_a,\ell_{a+1},\ldots,\ell_b)$.
    \end{lemma}
    \begin{proof}
      First suppose that there exists $k\in\bbo{i,j}$ such that $\ell_k=0$.
      Then $i$ and $j$ are non-adjacent in $\Gamma[\ell]$ by definition.
      Likewise, they are non-adjacent in the constraint graph of the
      permutation $\sigma=\sD(a,a+1,\ldots, b)$ since $k$ is a founder of
      $\sigma$ by Lemma~\ref{lem:bubEndpointCoding}.

      Otherwise, there exist $a'\leq i<j\leq b'$ such that $a'$ and $b'$ are
      consecutive zeros of $\ell$. Thus by definition of $\Gamma[\ell]$, the
      integers $i$ and $j$ are adjacent in $\Gamma[\ell]$ iff they are adjacent
      in the constraint graph of $\sigma'=\sD(a',a'+1,\ldots,b')$. Thus,
      the lemma will follow once we show that $i$ and $j$ are adjacent in
      $\Gamma_{\sigma}$ iff they are adjacent in $\Gamma_{\sigma'}$. Recall
      the explicit formula \eqref{eq:D} expressing the function $\sD$ as a
      composition of cycles, from the discussion preceding
      Lemma~\ref{lem:LDinv} in Section~\ref{sec:background}. From this
      formula, it follows that
      $$
        \sigma=\bigCircle_{k=a}^{i-1}\pi^-_{\bbb{k,k+\ell_k}}
          \circ \bigCircle_{k=i}^{j}\pi^-_{\bbb{k,k+\ell_k}}
          \circ \bigCircle_{k=j+1}^{b}\pi^-_{\bbb{k,k+\ell_k}}
      $$
      and
      $$
        \sigma'=\bigCircle_{k=a'}^{i-1}\pi^-_{\bbb{k,k+\ell_k}}
          \circ \bigCircle_{k=i}^{j}\pi^-_{\bbb{k,k+\ell_k}}
          \circ \bigCircle_{k=j+1}^{b'}\pi^-_{\bbb{k,k+\ell_k}}.
      $$
      The cycle $\pi^-_{\bbb{k,k+\ell_k}}$ leaves the relative ordering of
      $\bbb{i,j}$ unchanged whenever $k\not\in \bbb{i,j}$. Thus it is only
      the shared middle factor which determines the relative ordering of $\sigma$ and $\sigma'$ on
      $\bbb{i,j}$. This, in turn, determines whether
      $\sigma(i)<\sigma(k)>\sigma(j)$ for all $i<k<j$, which is equivalent
      to $i$ and $j$ being adjacent in $\sigma$, and respectively for $\sigma'$.
    \end{proof}

    \begin{cor}\label{cor:isGood}
      For $\ell\in\bbr{0,\infty}^{\m Z}$ with zero set unbounded from above
      and below,
      \begin{enumerate}[(i)]
        \item \label{cor:isGood:a} the set of bubble endpoints of
        $\Gamma[\ell]$ is equal to the zero set of $\ell$, and
        \item \label{cor:isGood:b} $\Gamma[\ell]$ is good.
        \end{enumerate}
    \end{cor}
    \begin{proof}
      By definition of $\Gamma[\ell]$, all zeros of $\ell$ are bubble
      endpoints. Conversely, suppose that $\ell_i>0$. Let $j<i$ be maximal
      such that $\ell_j=0$ and let $k>i$ be minimal such that $\ell_k=0$.
      Then $j$ is adjacent to $k$ in the constraint graph of
      $\sD(\ell_j,\ell_{j+1},\ldots,\ell_k)$, and therefore by
      Lemma~\ref{lem:localBub} $j$ and $k$ are also adjacent in
      $\Gamma[\ell]$. Thus $i$ is not a bubble endpoint, proving part
      \eqref{cor:isGood:a}.

      By part \eqref{cor:isGood:a} the bubble endpoints of $\Gamma[\ell]$
      are unbounded from above and below. By
      Lemma~\ref{lem:constraintGraphDecomp}, $\Gamma[\ell]$ decomposes
      into a collection of finite bubbles joined at their endpoints.
      Now Lemma~\ref{lem:colConst} implies that $\Gamma[\ell]$ is
      $q$-colorable. That consecutive bubble endpoints of $\Gamma[\ell]$ are
      adjacent follows from the corresponding property for a single bubble of
      the constraint graph of a permutation. Thus $\Gamma[\ell]$ is good.
    \end{proof}

    We can now prove the following key result.

    \begin{prop}\label{prop:colorInfiniteGraph}
      Set $u=\frac{q-1}{q-2}$ and let $L=(L_i)_{i\in\m Z}$ be an
      iid sequence of $u$-zero-weighted, $t$-geometric random variables.
      Conditional on $L$, choose a uniform proper $q$-coloring of
      $\Gamma[L]$. Then the (unconditional) law of the resulting coloring
      of $\m Z$ is $\MalCol_{q,t}$.
    \end{prop}

    We fix $u=\frac{q-1}{q-2}$ for the remainder of the section.

    \begin{proof}[Proof of Proposition~\ref{prop:colorInfiniteGraph}]
      By Corollary~\ref{cor:isGood}\eqref{cor:isGood:b}, $\Gamma[L]$ is a.s.\!
      good. Let $X=(X_i)_{i\in\m Z}$ be a uniform proper $q$-coloring of
      $\Gamma[L]$ and let $Y=(Y_i)_{i\in \Z}$ be a random coloring with law
      $\MalCol_{q,t}$. Let $\sigma_n$ be a random permutation of
      $\bbb{-n,n}$ with law $\BMal_{t,u}$ and let
      $Y^n=(Y^n_i\colon i\in \bbb{-n,n})$ be a uniform proper coloring of
      $\Gamma_{\sigma_n}$. By Proposition~\ref{prop:sec3main}, the sequence
      $Y^n$ is equal in law to $(Y_i\colon i\in\bbb{-n,n})$.

      In Proposition~\ref{prop:codingConvergence} we showed that
      $\sL(\sigma_n)$ converges in distribution to $L$. Thus by the
      Skorohod representation theorem \cite{kallenberg2006foundations},
      there exists a coupling of $(\sigma_n)_{n\geq 0}$ and $L$ such
      that $\sL(\sigma_n)$ a.s.\! converges to $L$. Fix such a coupling.

      Fix a finite interval $J$ of $\m Z$ and let $I$ be the (random) smallest
      interval containing $J$ whose endpoints are zeros of $L$.
      There is a random integer $N$ which is almost surely finite such
      that on the event $N<n$ we have that $\sL(\sigma_n)_i=L_i$ for
      all $i\in I$.

      It follows from our earlier observations that $(X_i)_{i\in I}$ and
      $(Y_i^n)_{i\in I}$ have the same conditional distribution given
      $L$, $\sigma_n$, and the event that $N<n$. Indeed, under this
      conditioning both $(X_i)_{i\in I}$ and $(Y_i^n)_{i\in I}$ are
      uniformly distributed on the set of proper $q$-colorings of the
      subgraph of $\Gamma[L]$ induced by $I$ by Lemmas~\ref{lem:unifQcol}
      and \ref{lem:localBub}. %, and  Corollary~\ref{cor:isGood}\eqref).
      Since $J \subseteq I$ and $N<\infty$ a.s., we deduce that
      $$
        (Y_i)_{i\in J} \eqd (Y^n_i)_{i \in J} \dto{n\to\infty} (X_i)_{i\in J}.
      $$
      The claim follows since $J$ was arbitrary.
    \end{proof}

    The last proposition yields the following construction of a random
    coloring with law $\MalCol$. Let $L$ be the above iid sequence. Assign
    to the zero set of $L$ a uniform proper $q$-coloring of $\m Z$.
    Conditional on these colors and $L$, assign to the intervals between each
    pair of consecutive zeros $i,j$ of $L$ an independent, uniformly random
    proper $q$-coloring of the constraint graph of $\sD(L_i,\ldots,L_j)$
    conditioned to agree at $i$ and $j$ with the colors previously assigned.
    Then by Corollary~\ref{cor:isGood}, the resulting coloring of $\m Z$ is
    conditionally a uniform proper $q$-coloring of the constraint graph of
    $\Gamma[L]$ given $L$, from which it follows by
    Proposition~\ref{prop:colorInfiniteGraph} that the coloring has law
    $\MalCol$. It remains to show, using this construction, that the colorings
    can be expressed both as finitary factors of iid processes and as
    functions of countable Markov chains. This is relatively routine, and
    we provide the details below.

  \begin{prop}\label{prop:Markov}
    There exists a countable state space $S$, a function $h:S \to \bbb{1,q}$,
    and a Markov process $(Y_i)_{i \in \Z}$ on $S$ such that the process
    $(X_i)_{i\in \Z} = \bigl(h(Y_i)\bigr)_{i\in \Z}$ has law $\MalCol_{q,t}$.
    Moreover the return time of each state of $S$ has exponential tail.
  \end{prop}
  \begin{proof}
    Let $L$ be an iid sequence of $u$-zero-weighted, $t$-geometric random variables. Conditional on $L$, let $X=(X_i)_{i\in\m Z}$ be a uniform proper $q$-coloring of $\Gamma[L]$. By the previous proposition, $X$ has law $\MalCol_{q,t}$. It remains to express $X$ as a function of a Markov process with the stated properties.

    For each $k \in \Z$, let $X^k$ and $L^k$ denotes the shifted sequences
    $$
      X^k=(X^k_i)_{i\in \Z} = (X_{i-k})_{i \in \Z} \qquad
      \qquad \text{ and } \qquad
      L^k=(L^k_i)_{i\in \Z} = (L_{i-k})_{i \in \Z}.
    $$
    For each $k$, let $f^+_k = \min \{i>0\colon L^k_i=0\}$ and $f^-_k=\max\{i\leq 0\colon L^k_i=0\}$. Let $G_k$ be the subgraph of $\Gamma[L^k]$ induced by $\bbb{f^-_k,f^+_k}$. For each $k\in \Z$, let $Y_k$ be the tuple
    \[Y_k = \left(-f^-_k, f^+_k, G_k, \bigl(X^k_j\bigr)_{j\in\bbb{f^-_k,f^+_k}}\right).\]
    The tuple $Y_k$ takes values in the set $S'$ of tuples $(f^1,f^2,G,x)$, where $f^1$ and $f^2$ are non-negative integers, $G$ is a graph with vertex set $\bbb{-f^1,f^2}$, and $x$ is a $q$-coloring of $G$. Note that $S'$ is countable. Let $S$ be the support of $Y_0$ on $S'$. We define $h:S\to \bbb{1,q}$ by setting $h(f^1,f^2,G,x)=x_0$, so that $(X_i)_{i\in \Z}=\bigl(h(Y_i)\bigr)_{i\in \Z}$ as desired.

    Clearly $Y$ is stationary. To prove that $Y$ is Markov, it suffices to show that $(Y_i)_{i>0}$ and $(Y_i)_{i<0}$ are conditionally independent given $Y_0$. Since $f_0^+$ is the location of the first bubble endpoint of $\Gamma[L]$ to the right of the origin, it follows from the definition of a uniform proper $q$-coloring that $(Y_i)_{i>f_0^+}$ is conditionally independent of $(Y_i)_{i<f_0^+}$ given $Y_{f_0^+}$. Now $Y_{f_0^+}$ determines the sequence $Y_0,Y_1,\ldots,Y_{f_0^+}$, as does $Y_0$, so therefore $(Y_i)_{i>0}$ and $(Y_i)_{i<0}$ are conditionally independent given $Y_0$. Thus $Y$ is a Markov process.

    That the return times have exponential tails follows in a straightforward manner.\qedhere
  \end{proof}

  \begin{prop}\label{prop:ffiid}
    There exists a ffiid process with law $\MalCol_{q,t}$ whose coding radius has exponential tail.
  \end{prop}

Before proving Proposition \ref{prop:ffiid}, we show how to produce a ffiid uniform proper $q$-coloring of $\Z$ with exponential tail on the coding radius, using a simple application of the technique of coupling from the past \cite{propp1998coupling}.

Let $(Z_i)_{i\in \Z}$ be an iid sequence, where each $Z_i$ is a chosen uniformly from the set $\{(a,b) \in \bbb{1,q}^2 : a\neq b\}$ of ordered pairs of distinct elements of $\bbb{1,q}$.

We claim that there is almost surely a unique sequence $X=(X_i)_{i\in \Z}$ satisfying the constraints
\begin{equation}\label{eq:zcolor} X_i = \begin{cases} Z_i^1 &\text{if } Z_i^1 \neq X_{i-1}\\
Z_i^2 &\text{if } Z_i^1 = X_{i-1}, \end{cases}\end{equation}
and furthermore that $X$ can be computed as a finitary factor of $Z$ with an exponential tail on the coding radius. Given its existence, it is easily seen that $X$ is a uniform proper $q$-coloring of $\Z$.

First, notice that if $Z_i^1 \notin \{Z_{i-1}^1,Z_{i-1}^2\}$, then we must have $X_i=Z_i^1$. Thus, it is possible to compute $X_i$ for arbitrary $i$ by first finding the maximal $T_i\leq i$ such that $Z_{T_i}^1 \notin \{Z_{T_i-1}^1,Z_{T_i-1}^2\}$, setting $X_{T_i}=Z^1_{T_i}$, and then computing $X_{j}$ for all $T_i \leq j \leq i$ by applying the recurrence \eqref{eq:zcolor}. This shows that there is almost surely a unique solution $X$ of \eqref{eq:zcolor}, and that $X$ can be computed as a finitary factor of $Z$ with coding radii $(i-T_i)_{i\in \Z}$. 
Finally, we observe that $i-T_i$ is a geometric random variable, since
\[\P(i-T_i \geq n) = \prod_{k=1}^n \P\Bigl(Z_{i-j+1}^1 \in \bigl\{Z_{i-j}^1,Z_{i-j}^2\bigr\}\Bigr)  = \left(\frac{2}{q}\right)^n. \]

\begin{proof}[Proof of Proposition \ref{prop:ffiid}]
  Consider the iid sequence $(Z_i,U_i,L_i)_{i\in \m Z}$ where $Z_i$ is chosen uniformly from the set $\{(a,b)\in \bbb{1,q}^2\colon a\not=b\}$, $U_i$ is chosen uniformly from $[0,1]$, and $L_i$ is a $u$-zero-weighted, $t$-geometric random variable. We construct the desired process in two steps. In the first step, we assign to $(i\in \m Z\colon L_i=0)$ a uniform proper $q$-coloring by applying the above procedure to $(Z_{i}\colon L_i=0)$.

  In the second step we assign colors to $(i\in \m Z\colon L_i>0)$. For such $i$ let $$f_i^-=\max \{j<i\colon L_j=0\}\quad\text{ and }\quad f_i^+=\min \{j>i\colon L_j=0\};$$ note that $i\in\bbo{f_i^-,f_i^+}$ and that $f_i^{\pm}$ were assigned colors in the previous step. Conditional on the previous step, let $X^{\bbb{f_i^-,f_i^+}}$ be a uniform proper $q$-coloring of the constraint graph of $\sD\bigl((\ell_k)_{k\in\bbb{f_i^-,f_i^+}}\bigr)$ consistent with the colors assigned to $f_i^{\pm}$. Assume that $X^{\bbb{f_i^-,f_i^+}}$ is defined on the probability space $[0,1]$ and assign to $i$ the color $X_i^{\bbb{f_i^-,f_i^+}}(U_{f_i^-})$, i.e., $U_{f_i^-}$ is used as a seed to generate the random coloring of $\bbb{f_i^-,f_i^+}$.

  It is easy to see that the coloring of $\m Z$ thus obtained is a finitary factor of the sequence $(Z_i,U_i,L_i)_{i\in\m Z}$ and that its coding radius has exponential tail. By Lemma~\ref{lem:unifQcol} it follows that, conditional on $L$, the coloring thus produced is a uniform proper $q$-coloring of $\Gamma[L]$. 
  Thus by  Proposition~\ref{prop:colorInfiniteGraph}, the coloring has law $\MalCol_{q,t}$.
\end{proof}

Recall the tuning equation \eqref{eqn:qkt}, which is
$$
 qt(t^k-1)=(t+1)(t^{k+1}-1).
$$

%%%%%%%%%%%%%%%%%%%%%%%%%%%
%%%%%%%%%%%%%%%%%%%%%%%%%%%
%%                       %%
%%       Proof of        %%
%%       main theorem    %%
%%                       %%
%%%%%%%%%%%%%%%%%%%%%%%%%%%
%%%%%%%%%%%%%%%%%%%%%%%%%%%

\begin{proof}[Proof of Theorem \ref{main}]
  Combine Propositions~\ref{prop:sec3main}, \ref{prop:reverseMain}, \ref{prop:Markov}, and \ref{prop:ffiid} to conclude that if for integers $q\geq 3$ and $k\geq 1$ there exists $t\in (0,1)$ satisfying the tuning equation \eqref{eqn:qkt}, then the theorem holds in the case $(k,q)$.

  That such a $t$ exists for $(k,q)=(1,5),\ (2,4),\ (3,3),$ and all larger $k$ and $q$ follows since these integers satisfy $qk>2(k+1)$, which implies that the polynomial
  $$
    qt(1-t^k)-(t+1)(1-t^{k+1})
  $$
  is negative at $t=0$, vanishes at $t=1$, and has negative derivative there. Thus it has a root in $(0,1)$, providing the desired solution of the tuning equation \eqref{eqn:qkt}.
\end{proof}

%%%%%%%%%%%%%%%%%%%%%%%%%%%
%%%%%%%%%%%%%%%%%%%%%%%%%%%
%%                       %%
%%       Section 7       %%
%%       painting and    %%
%%       conditioning    %%
%%%%%%%%%%%%%%%%%%%%%%%%%%%
%%%%%%%%%%%%%%%%%%%%%%%%%%%

\section{Painting algorithm and conditioning} \label{sec:paint}

  The primary purpose of this section is to verify correctness of the Painting Algorithm from the introduction.
  Recall that its input consists of positive integers $q$ and $k$ satisfying
  $qk>2(k+1)$ and its output is a random $q$-coloring of $\m Z$, which we claim is $k$-dependent.

  \begin{prop}\label{prop:finitary}\
    \begin{enumerate}[(i)]
      \item For all positive integers $q$ and $k$ satisfying $qk>2(k+1)$, there
      exists a unique $t=t(q,k)\in (0,1)$ such that $(q,k,t)$ satisfies the
      tuning equation \eqref{eqn:qkt}.
      \item The output of the Painting Algorithm has law
      $\MalCol_{q,t(q,k)}$.
    \end{enumerate}
  \end{prop}

  The other purpose of this section is establish that if one conditions the
  $1$-dependent $q$-coloring from the main theorem to only use colors in
  $\bbb{1,q-1}$,
  this results in the $2$-dependent $(q-1)$-coloring from the theorem, and no
  other pairs of colorings from the theorem are related in this manner.
  This is equivalent to the following.

  \begin{prop}\label{prop:condVals}
    The only pairs $(k,q)$ and $(k',q')$ such that there exists $t\in (0,1)$
    for which $(q,k,t)$ and $(q',k',t)$ both satisfy the tuning equation
    \eqref{eqn:qkt} are
    $$
      (k,q)=(1,q)\hspace{.5cm}\text{ and }\hspace{.5cm}(k',q')=(2,q-1),\qquad
      q\geq 5.
    $$
  \end{prop}

  First we establish the following properties of the tuning equation.
  \begin{lemma}\label{lem:qktSol}
    Let $q$ and $k$ be integers with $k>0$.
    \begin{enumerate}[(i)]
      \item \label{lem:qktSol:a} There exists $t\in (0,1)$ satisfying \eqref{eqn:qkt} if and only if $qk>2(k+1)$.
      \item \label{lem:qktSol:b} There is at most one $t\in[0,1)$ satisfying \eqref{eqn:qkt}.
    \end{enumerate}
  \end{lemma}
  We remark that this lemma and others in this section hold more generally
  when $q$ and $k$ are real-valued, with the same proofs.
  \begin{proof}
    The `if' direction of part \eqref{lem:qktSol:a} was established in the
    proof of Theorem~\ref{main}. For the `only if' direction, let
    $$
      f_{q,k}(t):=qt(1-t^k)-(t+1)(1-t^{k+1}).
    $$
    We show that when $q\leq q':=2(k+1)/k$, the function $f_{q,k}(t)$ has no zeros in $(0,1)$. Indeed, in this case $f_{q,k}(t)\leq f_{q',k}(t)$ and since
    \begin{equation}\label{eqn:secondDeriv}
      f_{q,k}''(t)=(k+1)t^{k-1}\bigl((k+2)t-(q-1)k\bigr),
    \end{equation}
    we have that $f_{q',k}''(t)=-(k+1)(k+2)(1-t)t^{k-1}$ and therefore $f_{q',k}$ is strictly concave on $(0,1)$. Combined with $f_{q',k}(1)=f'_{q',k}(1)=0$, it follows that $f_{q',k}(t)<0$ for all $t\in (0,1)$.
    Thus $f_{q,k}(t)<0$ when $q\leq 2(k+1)/k$ and $t\in (0,1)$, implying the `only if' direction of \eqref{lem:qktSol:a}.

    It remains to establish part \eqref{lem:qktSol:b} in the case when $q$ and $k$ satisfy $qk>2(k+1)$. Then $(q-1)t>k+2\geq (k+2)t$ for all $t\geq 1$, so $f_{q,k}''(t)<0$ for $t\in [0,1]$ by \eqref{eqn:secondDeriv}. This implies that $f_{q,k}$ has at most two zeros in $[0,1]$, counting the zero at $1$, so therefore $f_{q,k}$ has at most one zero in $[0,1)$.
  \end{proof}

\begin{proof}[Proof of Proposition~\ref{prop:finitary}]
  The first part of the proposition follows from the Lemma~\ref{lem:qktSol}.

  For the second part, let $\Gamma_{\text{alg}}$ be the graph on $\m Z$ in which integers $i<j$ are adjacent iff $i$ and $j$ are the first two elements of $\bbb{i,j}$ assigned colors by the algorithm. It is straightforward to verify that the bubble endpoints of $\Gamma_{\text{alg}}$ are the integers assigned colors during Stage 1 of the algorithm, and that conditional on $\Gamma_{\text{alg}}$, the coloring $X$ is a uniform proper $q$-coloring of $\Gamma_{\text{alg}}$.

  Thus by Proposition~\ref{prop:colorInfiniteGraph}, it suffices to show that $\Gamma_{\text{alg}}$ is equal in law to $\Gamma[L]$ where $L=(L_i)_{i\in\m Z}$ is an iid sequence of $u$-bubble-weighted, $t$-geometric random variables. The bubble endpoints of the two graphs have the same law by Corollary~\ref{cor:isGood}\eqref{cor:isGood:a}, since
  $$
    \m P(L_i\geq 1)=\frac{t/(1-t)}{u+t/(1-t)}=\frac{t(q-2)}{q-1-t}=s,
  $$
  where we have substituted $u=\frac{q-1}{q-2}$ in the second equality.

  The result will follow from the claim that, for all integers $i<j$, the subgraphs of $\Gamma[L]$ and $\Gamma_{\text{alg}}$ induced by $\bbb{i,j}$ have the same conditional law given that $i$ and $j$ are bubble endpoints of $\Gamma[L]$ and $\Gamma_{\text{alg}}$, respectively.
  We establish this claim by showing that both conditional laws are equal to the law of the constraint graph of a Mallows-distributed permutation of $\bbb{i,j}$.
  For $\Gamma[L]$ this is straightforward to verify using Lemma~\ref{lem:localBub} together with the observation that the conditional law of a $u$-zero-weighted, $t$-geometric random variable conditioned to be positive does not depend on $u$.
  To prove the same for $\Gamma_{\text{alg}}$, we consider the permutation $\sigma$ of $\bbb{i,j}$ for which $\sigma(i)-i$ equals the number of elements of $\bbb{i,j}$ that have been assigned colors prior to $i$. We regard $\sigma(i)$ as the arrival time of $i$. On the event that $i$ and $j$ are bubble endpoints of $\Gamma_{\text{alg}}$, the subgraph of $\Gamma_{\text{alg}}$ induced by $\bbb{i,j}$ is seen to coincide with the constraint graph of $\sigma$.

  The integer with arrival time $k$ appears in position $\sLtilde(\sigma)_k$ from the right in the subsequence of $\bbb{i,j}$ consisting of integers that have previously arrived.
  Thus, the random choices made by the algorithm ensure that:
  \begin{itemize}
    \item $\sLtilde(\sigma)_k$ is $(k-i)$-truncated $t$-geometric, and
    \item $\bigl(\sLtilde(\sigma)_k\bigr)_{k\in\bbb{i,j}}$ is a sequence of independent random variables.
  \end{itemize}
  Thus by Lemma~\ref{lem:LehmerBub}, the permutation $\sigma$ has law $\Mal_t$.
  It follows that for all integers $i<j$ the conditional laws of the subgraphs of $\Gamma[L]$ and $\Gamma_{\text{alg}}$ induced by $\bbb{i,j}$, given that $i$ and $j$ are bubble endpoints of the respective graphs, coincide.

  From this we deduce that the conditional laws of $\Gamma[L]$ and $\Gamma_{\text{alg}}$ given their respective sets of bubble endpoints are the same. Since we have previously established that these sets have the same distribution, it follows that $\Gamma[L]$ and $\Gamma_{\text{alg}}$ are equal in law.
\end{proof}

In the introduction, we claimed that if one conditions on the absence of color $q$ in the $1$-dependent $q$-colorings from the main theorem, the resulting $(q-1)-$coloring is $2$-dependent, and that these are the only pairs of colorings in the theorem related to one another by conditioning in this manner. This follows from the properties of the tuning equation \eqref{eqn:qkt} established in the following lemma.

\begin{lemma}\label{lem:qktSolutions2}
  Consider positive integers $q$ and $k$ and a real number $t\in(0,1)$.
  \begin{enumerate}
    \item \label{lem:qktSolution2:a} The triple $(q,1,t)$ satisfies \eqref{eqn:qkt} if and only if $(q-1,2,t)$ does.
    \item \label{lem:qktSolution2:b} Given $(q,t)$, there is at most one $k$ satisfying \eqref{eqn:qkt}.
    \item \label{lem:qktSolution2:c} Suppose that $qk>2(k+1)$ and let $t(q,k)$ denote the unique solution of \eqref{eqn:qkt} satisfying $0<t(q,k)<1$. Then $t(q,k)>\frac{1}{q-1}$ and $\lim_{k\to\infty}t(q,k)=\frac{1}{q-1}$.
    \item \label{lem:qktSolution2:d} Given $t$ with $0 < t<1$, there is at most one pair of integers $(q,k)$ with $q>1$ and $k\geq 2$ satisfying \eqref{eqn:qkt}.
  \end{enumerate}
\end{lemma}
\begin{proof}
  Part \ref{lem:qktSolution2:a} is an easy calculation. Part \ref{lem:qktSolution2:b} follows since \eqref{eqn:qkt} is equivalent to $$t^{k+1}=\frac{(q-1)t-1}{q-1-t},$$ and this also implies part \ref{lem:qktSolution2:c}.

  Finally we establish part \ref{lem:qktSolution2:d}. Solving for $q$ in
  \eqref{eqn:qkt} yields
  \begin{equation}\label{eqn:solveForQ}
    q=\left(1+\frac{1}{t}\right)\left(t+\frac{1-t}{1-t^k}\right).
  \end{equation}
  The right side of \eqref{eqn:solveForQ} is strictly decreasing in $k$. \
  Furthermore as $k\to\infty$ it tends to $\frac{1}{t}+1$, whereas when $k=2$
  it evaluates to $\frac{1}{t}+1+t$. Therefore if $(q,k,t)$ is any solution of
  \eqref{eqn:solveForQ} with $q>1$ and $k\geq 2$, we must have that
  \begin{equation}\label{eqn:boundForQ}
    \frac{1}{t}+1<q\leq \frac{1}{t}+1+t,\qquad k>2.
  \end{equation}
  Since $t<1$ there can be at most one integer $q$ satisfying
  \eqref{eqn:boundForQ}, in which case by part \ref{lem:qktSolution2:b} there is
  at most one $k$ satisfying \eqref{eqn:qkt}.
\end{proof}

%%%%%%%%%%%%%%%%%%%%%%%%%%%
%%%%%%%%%%%%%%%%%%%%%%%%%%%
%%                       %%
%%       Proof of        %%
%%       bit fin factor  %%
%%                       %%
%%%%%%%%%%%%%%%%%%%%%%%%%%%
%%%%%%%%%%%%%%%%%%%%%%%%%%%

\section{Bit-finitary factors}\label{sec:bitproof}

In this section we prove Theorem \ref{thm:bit}.

\begin{proof}[Proof of Theorem \ref{thm:bit}]
Let $F$ and $U$ be as in the statement of the theorem.
For each $k \geq 1,$ let
\[p(k)=\P\bigl(r_0(U) \leq k\bigr)\] be the probability that $F(U)_0$ is determined by the restriction of $U$ to $[-k,k]\times [0,d]$ for some $d\geq 0$. 
For each $k\geq 1$,   let $d(k)$ be the minimal value of $d$ such that the probability that
$F(U)_0$ is determined by the restriction of $U$ to $[-k,k] \times [0,d]$ is at least $p(k)-2^{-k}$. In particular, $d(k)=0$ if $p(k) \leq 2^{-k}$.
Let $R_i(U)$ be minimal such that $F(U)_i$ is determined by the restriction of $U$ to
\[S_i(U):=\bigl[i-R_i(U),i+R_i(U)\bigr] \times \left[0,d\bigl(R_i(U)\bigr)\right].\]
The definition of $r_i(U)$ and $R_i(U)$ ensure that
\[\P\bigl(R_0(U) \geq k\bigr) \leq \P\bigl(r_0(U) \geq k\bigr)+2^{-k}\]
for every $k\geq 0$ and hence that
\[\E[R_0(U)] = \sum_{k\geq 1} \P(R_0(U) \geq k) \leq \E[r_0(U)]+1 <\infty.\]

 For each $i\in \Z$, we define
\[T^+_i(U) = \sup\left\{j \geq i :\, R_k(U) \geq j-k \text{ for some $k\leq i$}\right\}\]
and
\[T^-_i(U) = \inf\left\{j \leq i :\, R_k(U) \geq k-j \text{ for some $k\geq i$}\right\}.\]
Similarly, for each $i,j \in \Z$ with $j \leq T^+_i$, we define
\[D_{i,j}(U) = \inf \left\{D \geq 0: S_k \cap (\N\times\{j\}) \subseteq [0,D] \times \{j\} \text{ for all $k\leq i$}\right\}
\]
Since $\E[R_0(U)]$ is finite, it follows from Borel-Cantelli that for each $i \in \Z$ there exist at most finitely many $k \in \Z$ for which $R_k(U) \geq |k-i|$ a.s., and this easily implies that $T^+_i(U)$, $T^-_i(U)$, and $D_{i,j}(U)$ are finite for every $i\in \Z$ and every $j\leq T^+_i$ a.s.

Let $\mathcal{K}$ be the set of pairs $(K,y)$, where $K$ is a finite subset $\Z \times \N$ and $y$ is a function from $K$ to $\{0,1\}$. Clearly $\mathcal{K}$ is countable.
For each $i \in \Z$, we define a finite set $K_i(U) \subset \Z\times \N$ by
\[K_i(U) =\left\{(j,k) \in \Z \times \N : j \in \big[T_0^-(U), T_0^+(U)\big],\; k \in \big[0, D_{i,j}(U)\big]\right\},\]
and let
\[ \bar K_i(U) =\left\{(j,k) \in \Z \times \N : j \in \big[T_0^-(U)-i, T_0^+(U)-i\big],\; k \in \big[0, D_{i,j}(U)\big]\right\}.\]
We define a factor $G: \{0,1\}^{\Z\times \N}\to \mathcal{K}^\Z$ by setting
\[G_i(U)=\left( \bar K_i(U), U|_{K_i(U)}\right).\]
Fix an element $a_0\in A$ arbitrarily, and define $h:\mathcal{K}\to A$ by letting $h(K,y)=a$ if $a\in A$ is such that $F(x)_0=a$ for a.e.\ $x$ such that $x|_K=y$, and letting $h(K,y)=a_0$ if no such $a\in A$ exists.
The construction of $G$ ensures that the value of $F(U)_0$ is determined by the restriction of $U$ to $K_0(U)$, and it follows that $H \circ G = F$, where $H:\mathcal{K}^\Z\to A^\Z$ is the factor
\[H_i\bigl((K_j,y_j)_{j\in \Z}\bigr) = h(K_i,y_i).\]
It remains to prove only that the process $\bigl( G_i(U)\bigr)_{i\in\Z}$ is a Markov chain. To see this,
observe that the subsets of $\Z\times\N$ that are queried in order to compute $\bigl( G_i(U) \bigr)_{i\geq1}$ and $\bigl( G_i(U) \bigr)_{i \leq -1}$ have intersection contained in the subset of $\Z \times \N$ that is queried to compute $G_0(U)$. It follows that
 $\bigl( G_i(U) \bigr)_{i \geq 1}$ and $\bigl( G_i(U) \bigr)_{i \leq -1}$ are conditionally independent given $G_0(U)$, so that  $\bigl( G_i(U)\bigr)_{i\in\Z}$ is indeed a Markov chain. \qedhere

\end{proof}

\section{Compact Markov chains}\label{sec:compact}

By a {\bf process} on a space $S$, we mean a random element of $S^{\m Z}$ that is measurable with respect to the product Borel $\sigma$-algebra.
As mentioned in the introduction, it is an open question whether there exists a stationary Markov process $(X_i)_{i\in\m Z}$ on a compact metric space $(S,d)$, an integer $k>0$, and a real number $\ve>0$ such that
\begin{enumerate}
  \item $X_0$ and $X_k$ are independent, and
  \item $d(X_0,X_1)\geq \ve$ almost surely.
\end{enumerate}
It is trivial to construct chains that satisfy either one of the two conditions. Here are two somewhat interesting examples. Firstly, let $S=[0,1]^2$. Conditional on $X_0=(u,v)$, let $X_1=(v,U)$, where $U$ is uniformly distributed on $[0,1]$. This satisfies (i) with $k=2$, but not (ii). Secondly, let $S$ be the unit sphere $\{x\in \m R^3\colon \|x\|_2=1\}$. Conditional on $X_0$, let $X_1$ be uniformly distributed on the circle $\{y\in S\colon \langle y,X_0\rangle =0\}$ (where $\langle\cdot,\cdot\rangle$ is the standard inner product on $\m R^3$). This satisfies (ii) (with the Euclidean metric and $\ve=1$) but not (i).

\begin{proof}[Proof of Proposition \ref{prop:noReversibleCompactChain}]
  Let $\pi$ denote the law of $X_0$. Since $S$ is Polish, there is a Markov kernel that is a regular conditional probability for $\m P(X_1\in A\mid X_0)$. This gives rise to a Markov transition operator $P$ on $L^2(S,\pi)$. Since $X$ is reversible, $P$ is self-adjoint.

  Since $X_0$ is independent of $X_k$, it follows that $P^{i}=P^k$ for all $i\geq k$. We claim that if $k\geq 2$, then $P^k=P^{k-1}$. Indeed, for all $f\in L^2(S,\pi)$,
  \begin{align}\label{eq:markovOperatorProj}
    \|P^{k-1}f-P^kf\|^2&=\|P^{k-1}f\|^2-2\langle P^{k-1}f,P^kf\rangle +\|P^{k}f\|^2\nonumber\\
    &=
    \langle f,P^{2k-2}f\rangle -2\langle f,P^{2k-1}f\rangle+\langle f,P^{2k}f\rangle.
  \end{align}
  Since $k\geq 2$, we have that $2k-2\geq k$ and therefore $P^{2k-2}=P^{2k-1}=P^{2k}$. Thus \eqref{eq:markovOperatorProj} vanishes for all $f$, so $P^{k-1}=P^k$. Hence by induction it follows that $P=P^2$.

  Consider a cover of $S$ by balls of radius $\ve/4$. Fix a partition of unity $\{f_i\}\subset L^2(S,\pi)$ subordinate to this cover. Then $Pf_i$ and $f_i$ have disjoint support for all $i$. Thus $\langle Pf_i,f_i\rangle=0$, implying that
  $$
    \|Pf_i\|^2=\langle Pf_i,Pf_i\rangle=\langle Pf_i,f_i\rangle=0.
  $$
  Consequently $\sum_i Pf_i=0$, so $P$ maps the constant $1$ function to the zero function. But this contradicts the fact that $P$ is a Markov transition operator.
\end{proof}

As explained in the introduction, any process satisfying the above conditions yields a finitely dependent coloring. Schramm established that no hidden-Markov finitely dependent $q$-coloring exists, for any $q$. (This was first published in \cite[Proposition 3]{HL}.) Combining this with the observation in the previous paragraph implies that there is no \emph{finite} Markov chain satisfying the conditions of the above question. %Question \ref{question:1}.
 On the other hand, any stationary stochastic process can trivially be expressed as a function of an \emph{uncountable}-state Markov chain.

\begin{proof}[Proof of Proposition \ref{prop:existsNonCompact}]
  We first describe the process $(X_n)_{n\in\Z}$, and then explain why it is Markov.
  Let $(\omega_{i,j}\colon i\in\m Z,\ j\in \m N)$ be iid Bernoulli$(\tfrac12)$ random bits indexed by the discrete half plane.
  The process $(X_n)_{n\in\m Z}$ will be a deterministic function of these bits.
  Let
  $$
    h(n):=\min\bigl\{j\in\N:\omega_{n-1,j}\neq\omega_{n,j}\bigr\},\qquad n\in\m Z,
  $$
  be the height of the first discrepancy between the bits in columns $n-1$ and $n$.
  Note that $h(n)$ is a.s.\ finite.
  Let $X_n$ be the $2$-by-$h(n)$ matrix
  $$
    \Bigl(\omega_{i,j}: (i,j)\in \{n-1,n\}\times\bigl\{1,\ldots, h(n)\bigr\} \Bigr),
  $$
  consisting of bits in the two columns up to the discrepancy.
  The state space $S$ is the set of binary matrices of width $2$ with the two columns differing exactly in the last row; note that this set is countable.
  We equip $S$ with the discrete metric $d(x,y):=\mathbbm{1}[x\neq y]$.

  We now show that the process $(X_n)_{n\in\m Z}$ is Markov.
  Consider the $\sigma$-algebras
  $$
    \mathcal F_n=\sigma\bigl((\omega_{n-1,j},\omega_{n,j})\colon j\in\mathbb N\bigr),\qquad n\in\m Z.
  $$
  By considering cylinder events, it is easily verified that the sigma algebras \[\sigma(\mathcal F_{n+1},\mathcal F_{n+2}\ldots) \qquad \text{ and } \qquad \sigma(\mathcal F_{n-1},\mathcal F_{n-2}\ldots)\] are conditionally independent given $\mathcal F_n$.
  Since $X_m$ is measurable with respect to $\mathcal F_m$ for all $m\in\m Z$, the Markov property follows.

  Clearly the sequence $(X_n)_{n\in\m Z}$ is stationary.
  Furthermore $X_0$ is independent of $X_2$, since they are functions of disjoint sets of independent bits.
  All that remains is to verify that $X_n\not=X_{n+1}$.
  Suppose to the contrary that $X_n=X_{n+1}$.
  Then $h(n)=h(n+1)$; call the common value $j$.
  Since $X_n=X_{n+1}$ we have that $(\omega_{n-1,j},\omega_{n,j})=(\omega_{n,j},\omega_{n+1,j})$.
  But since $h(n)=j$ we have that $\omega_{n-1,j}\neq\omega_{n,j}$.
  Thus we have derived a contradiction.
\end{proof}

\section{Higher dimensions and shifts of finite type}\label{sec:highDim}
\begin{proof}[Proofs of Corollaries \ref{cor:highDim} and \ref{cor:shifts}]
  A simple modification of the proofs of \cite[Corollaries 5 and 6]{HL} establishes Corollaries \ref{cor:highDim} and \ref{cor:shifts}, respectively. Namely, replace the 1-dependent 4-coloring used in the proof of Corollary 20 of \cite{HL} with the ffiid 1-dependent 5-coloring with exponential tail on the coding radius from Theorem~\ref{main}. This results in an ffiid process with exponential tails, since the maximum of a finite (deterministic) number of independent random variables with exponential tails still has exponential tails.
\end{proof}

\section{Open problems}\label{sec:open}

  \begin{enumerate}
    \item For the pairs $(k,q)=(1,4)$ and $(2,3)$, can the $k$-dependent $q$-coloring be expressed as a finitary factor of iid with finite expected coding radius? (We suspect not.)
    \item For which pairs $(k,q)$ do there exist other color-symmetric $k$-dependent $q$-colorings of $\m Z$ besides the Mallows colorings?  We believe that there are no others for the pairs $(1,4)$ and $(2,3)$.
    \item Are there finitary factors of iid that are not expressible as bit-finitary factors of iid?
    \item
    Does there exist a stationary Markov process $(X_i)_{i\in\m Z}$ on a compact metric space $(S,d)$, an integer $k>0$, and a real number $\ve>0$ such that $X_0$ and $X_k$ are independent and $d(X_0,X_1)\geq \ve$ almost surely?
    \item For every word $x\in \m Z^n$ and every $0\leq i\leq \binom{n}{2}$, is there a `natural' bijection between proper buildings of $x$ having $i$ and $\binom{n}{2}-i$ inversions?
  \end{enumerate}

\section*{Acknowledgements}\label{sec:ack}
  AL and TH were supported by internships at Microsoft Research while portions of this work were completed. TH was also supported by a Microsoft Research PhD fellowship.

  \bibliographystyle{habbrv}
  \bibliography{coloring}

\end{document}